\documentclass[mnsc,nonblindrev]{optonline} 

\usepackage{epstopdf}
\usepackage{bm}
\usepackage{cases}
\usepackage{paralist}
\usepackage{empheq}
\usepackage{multirow}
\usepackage{subfig}
\usepackage{subfloat}
\usepackage[table]{xcolor}
\usepackage{enumitem}
\usepackage{capt-of}
\usepackage{algorithm}
\usepackage{url}
\usepackage[noend]{algpseudocode}
\usepackage{float}
\usepackage{caption}
\usepackage{wrapfig}

\OneAndAHalfSpacedXI
\usepackage{natbib}
 \bibpunct[, ]{(}{)}{,}{a}{}{,}%
 \def\bibfont{\small}%

\TheoremsNumberedThrough     
\ECRepeatTheorems
\EquationsNumberedThrough    
\MANUSCRIPTNO{MS-14-00887}

\renewcommand{\qed}{\hfill $\square$}

\newcommand{\milp}{\mathcal{P}}	
\newcommand{\milpo}{\mathcal{P}^{\text{o}}}								
\newcommand{\rlp}{\mathcal{R}_{\text{LP}}}
\newcommand{\rlpo}{\mathcal{R}_{\text{LP}}^{\text{o}}}
\newcommand{\rlpob}{\mathcal{R}_{\text{LP}}^{\text{ob}}}
\newcommand{\rlpb}{\mathcal{R}^{\text{b}}_{\text{LP}}}					
\newcommand{\rsdp}{\mathcal{R}_{\text{SDP}}}
\newcommand{\rsdpo}{\mathcal{R}^{\text{o}}_{\text{SDP}}}
\newcommand{\rsdpob}{\mathcal{R}^{\text{ob}}_{\text{SDP}}}			

\begin{document}

\RUNTITLE{Size Matters: Cardinality-Constrained Clustering and Outlier Detection via Conic Optimization}
\RUNAUTHOR{Rujeerapaiboon, Schindler, Kuhn, Wiesemann}
\TITLE{Size Matters: Cardinality-Constrained Clustering and Outlier Detection via Conic Optimization}

\ARTICLEAUTHORS{%
\AUTHOR{Napat Rujeerapaiboon}
\AFF{\textit{Department of Industrial Systems Engineering and Management, National University of Singapore, Singapore}, \EMAIL{isenapa@nus.edu.sg}} 
\AUTHOR{Kilian Schindler, Daniel Kuhn}
\AFF{\textit{Risk Analytics and Optimization Chair, \'Ecole Polytechnique F\'ed\'erale de Lausanne, Switzerland}, \\ \EMAIL{kilian.schindler@epfl.ch}, \EMAIL{daniel.kuhn@epfl.ch}} 
\AUTHOR{Wolfram Wiesemann}
\AFF{\textit{Imperial College Business School, Imperial College London, United Kingdom,} \EMAIL{ww@imperial.ac.uk}}
}

\ABSTRACT{%
{\color{black}Plain vanilla $K$-means clustering has proven to be successful in practice, yet it suffers from outlier sensitivity and may produce highly unbalanced clusters.} To mitigate both shortcomings, we formulate a joint outlier detection and clustering problem, which assigns a prescribed number of datapoints to an auxiliary outlier cluster and performs cardinality-constrained \mbox{$K$-means} clustering on the residual dataset{\color{black}, treating the cluster cardinalities as a given input}. We cast this problem as a mixed-integer linear program (MILP) that admits tractable semidefinite and linear programming relaxations. We propose deterministic rounding schemes that transform the relaxed solutions to feasible solutions for the MILP. We also prove that these solutions are optimal in the MILP if a cluster separation condition holds. 
}

\KEYWORDS{Semidefinite programming, $K$-means clustering, outlier detection, optimality guarantee}

\maketitle

\section{Introduction}

Clustering aims to partition a set of datapoints into a set of clusters so that datapoints in the same cluster are more similar to another than to those in other clusters. Among the myriad of clustering approaches from the literature, $K$-means clustering stands out for its long history dating back to 1957 as well as its impressive performance in various application domains, ranging from market segmentation and recommender systems to image segmentation and feature learning \cite[]{Jain2010651}.

{\color{black}This paper studies the \emph{cardinality-constrained $K$-means clustering problem}, which we define as the task of partitioning $N$ datapoints $\bm{\xi}_1, \ldots, \bm{\xi}_N \in \mathbb{R}^d$ into $K$ clusters $I_1, \ldots, I_K$ of prescribed sizes $n_1, \ldots, n_K$, with $n_1 + \ldots + n_K = N$, so as to minimize the sum of squared intra-cluster distances. We can formalize the cardinality-constrained $K$-means clustering problem as follows,
\begin{equation}\label{opt:kmean_size}
\begin{array}{l@{\quad}l}
\displaystyle \text{minimize} & \textstyle \sum_{k=1}^K \sum_{i \in I_k} \lVert \bm{\xi}_i - \frac{1}{n_k} ( \sum_{j \in I_k} \bm{\xi}_j ) \rVert^2 \\[3mm]
\displaystyle \text{subject to} & \displaystyle (I_1, \ldots, I_K) \in \mathfrak{P} (n_1, \ldots, n_K),
\end{array}
\end{equation}
where
\begin{equation*}
\mathfrak{P} (n_1, \ldots, n_K) = \Big\{ (I_1, \ldots, I_K) \, : \, \vert I_k \vert = n_k \;\; \forall k, \;\; \cup_{k=1}^K I_k = \{ 1, \ldots, N \}, \;\; I_k \cap I_\ell = \emptyset \;\; \forall k \neq \ell \Big\}
\end{equation*}
denotes the ordered partitions of the set $\{ 1, \ldots, N \}$ into $K$ sets of sizes $n_1, \ldots, n_K$, respectively.}

Our motivation for studying problem~\eqref{opt:kmean_size} is threefold. Firstly, it has been shown by \cite{Bennett2000} and \cite{NIPS2005_2902} that the algorithms commonly employed for the \emph{un}constrained $K$-means clustering problem frequently produce suboptimal solutions where some of the clusters contain very few or even no datapoints. In this context, cardinality constraints can act as a regularizer that avoids local minima of poor quality. Secondly, many application domains require the clusters $I_1, \ldots, I_K$ to be of comparable size. This is the case, among others, in distributed clustering (where different computer clusters should contain similar numbers of network nodes), market segmentation (where each customer segment will subsequently be addressed by a marketing campaign) and document clustering (where topic hierarchies should display a balanced view of the available documents), see \cite{Banerjee2006} and \cite{NIPS2013_5096}. Finally, and perhaps most importantly, $K$-means clustering is highly sensitive to outliers. To illustrate this, consider the dataset in Figure~\ref{fig:outliers}, which accommodates three clusters as well as three individual outliers. The $K$-means clustering problem erroneously merges two of the three clusters in order to assign the three outliers to the third cluster (top left graph), whereas a clustering that disregards the three outliers would recover the true clusters and result in a significantly lower objective value (bottom left graph). The cardinality-constrained $K$-means clustering problem, where the cardinality of each cluster is set to be one third of all datapoints, shows a similar behavior on this dataset (graphs on the right). We will argue below, however, that the cardinality-constrained $K$-means clustering problem~\eqref{opt:kmean_size} offers an intuitive and mathematically rigorous framework to robustify $K$-means clustering against outliers. {\color{black} A comprehensive and principled treatment of outlier detection methods can be found in the book of \cite{Aggarwal2013}.}

\begin{figure}[h]
 	\centering
 	\hspace{-7mm}
	\includegraphics[width=0.40\paperwidth]{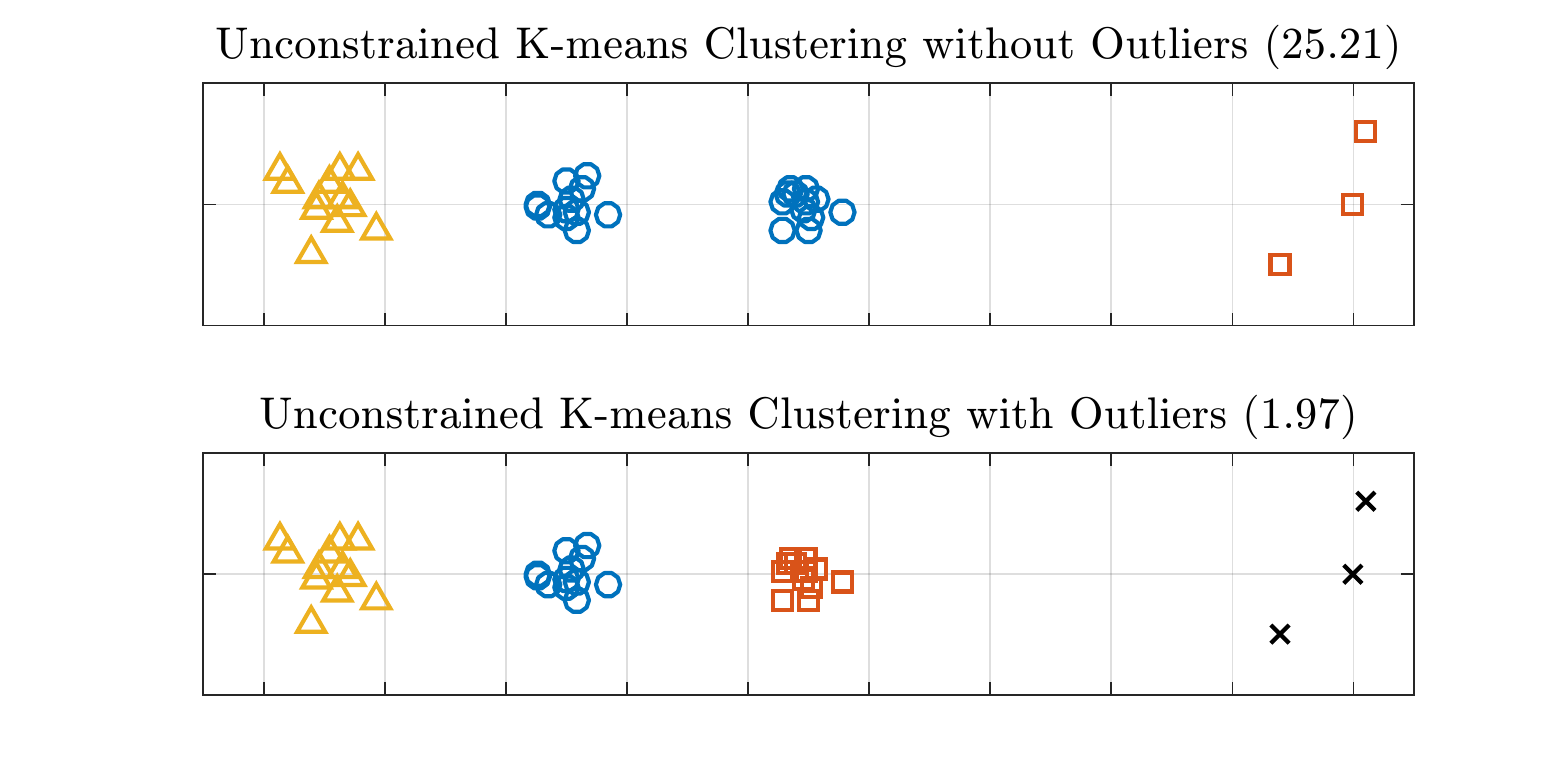} \hspace{-7mm}
	\includegraphics[width=0.40\paperwidth]{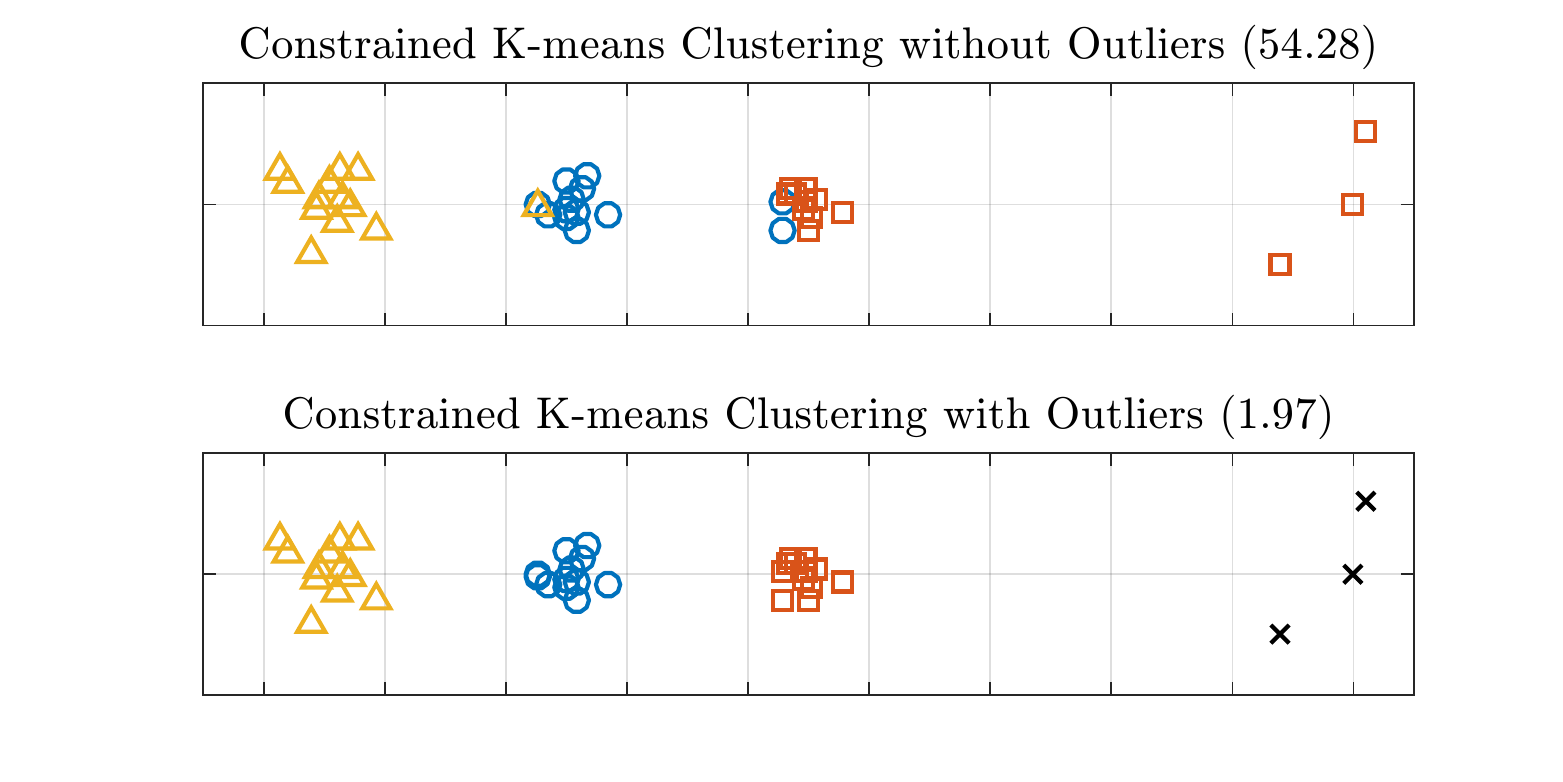} 
	\caption{Sensitivity of the (un)constrained $K$-means clustering problem to outliers. Indicated in parentheses next to the panel titles are the respectively achieved sums of squared intra-cluster distances.}
	\label{fig:outliers}
\end{figure}

To our best knowledge, to date only two solution approaches have been proposed for problem~\eqref{opt:kmean_size}. \cite{Bennett2000} combine a classical local search heuristic for the unconstrained $K$-means clustering problem due to \cite{Lloyd1982} with the repeated solution of linear assignment problems to solve a variant of problem~\eqref{opt:kmean_size} that imposes lower bounds on the cluster sizes $n_1, \ldots, n_K$. \cite{Banerjee2006} solve the balanced version of problem~\eqref{opt:kmean_size}, where $n_1 = \cdots = n_K$, by sampling a subset of the datapoints, performing a clustering on this subset, and subsequently populating the resulting clusters with the remaining datapoints while adhering to the cardinality constraints. {\color{black}Balanced clustering is also considered by \cite{Malinen2014} and \cite{Costa2017}. \cite{Malinen2014} proceed similarly to \cite{Bennett2000} but take explicit advantage of the Hungarian algorithm to speed up the cluster assignment step within the local search heuristic. \cite{Costa2017} propose a variable neighbourhood search heuristic that starts from a random partition of the datapoints into balanced clusters and subsequently searches for better solutions in the neighbourhood obtained by an increasing number of datapoint swaps between two clusters. Although all of these heuristics tend to quickly produce solutions of high quality, they are not known to be polynomial-time algorithms, they do not provide bounds on the suboptimality of the identified solutions, and their performance may be sensitive to the choice of the initial solution.} Moreover, neither of these local search schemes accommodates for outliers.

In recent years, several conic optimization schemes have been proposed to alleviate the shortcomings of these local search methods for the unconstrained $K$-means clustering problem \cite[]{Peng2007, Awasthi2015}. \cite{Peng2007} develop two semidefinite programming relaxations of the unconstrained $K$-means clustering problem. Their weaker relaxation admits optimal solutions that can be characterized by means of an eigenvalue decomposition. They further use this eigenvalue decomposition to set up a modified $K$-means clustering problem where the dimensionality of the datapoints is reduced to $K-1$ (provided that their original dimensionality was larger than that). To obtain an upper bound, they solve this $K$-means clustering problem of reduced dimensionality, which can be done either exactly by enumerating Voronoi partitions, as described in \cite{Inaba1994}, or by approximation methods such as those in \cite{hasegawa1993cient}. Using either approach, the runtime grows polynomially in the number of datapoints~$N$ but not in the number of desired clusters $K$. Hence, this method is primarily suitable for small~$K$. Similar conic approximation schemes have been developed by \cite{Elhamifar2012} and \cite{Nellore2015165} in the context of unconstrained exemplar-based clustering.

\cite{Awasthi2015} and \cite{Iguchi2015} develop probabilistic recovery guarantees for the stronger semidefinite relaxation of \cite{Peng2007} when the data is generated by a stochastic ball model (\emph{i.e.,} datapoints are drawn randomly from rotation symmetric distributions supported on unit balls). More specifically, they use primal-dual arguments to establish conditions on the cluster separation under which the semidefinite relaxation of \cite{Peng2007} recovers the underlying clusters with high probability as the number of datapoints $N$ increases. The condition of \cite{Awasthi2015} requires less separation in low dimensions, while the condition of \cite{Iguchi2015} is less restrictive in high dimensions. In addition, \cite{Awasthi2015} consider a linear programming relaxation of the unconstrained $K$-means clustering problem, and they derive similar recovery guarantees for this relaxation as well.

Two more papers study the recovery guarantees of conic relaxations under a stochastic block model (\emph{i.e.,} the dataset is characterized by a similarity matrix where the expected pairwise similarities of points in the same cluster are higher than those of points in different clusters). \cite{Ames2014} considers the densest $K$-disjoint-clique problem whose aim is to split a given complete graph into $K$ subgraphs such as to maximize the sum of the average similarities of the resulting subgraphs. $K$-means clustering can be considered as a specific instance of this broader class of problems. By means of primal-dual arguments, the author derives conditions on the means in the stochastic block model such that his semidefinite relaxation recovers the underlying clusters with high probability as the cardinality of the smallest cluster increases. \cite{VinayakHassib2016} develop a semidefinite relaxation and regularize it with the trace of the cluster assignment matrix. Using primal-dual arguments they show that, for specific ranges of the regularization parameter, their regularized semidefinite relaxation recovers the true clusters with high probability as the cardinality of the smallest cluster increases. The probabilistic recovery guarantees of \cite{Ames2014} and \cite{VinayakHassib2016} can also be extended to datasets containing outliers.

\begin{table}[H]
\scriptsize
\centering
\begin{tabular}{r||c|c|c|c|c}
    & Awasthi et al. & Iguchi et al. & Ames & Vinayak and Hassibi & This Paper\\ \hline 
	data generating model & stochastic ball & stochastic ball & stochastic block & stochastic block & none/arbitrary \\
	type of relaxation & SDP + LP & SDP & SDP & SDP & SDP + LP \\
	type of guarantee & stochastic & stochastic & stochastic & stochastic & deterministic \\
	guarantee depends on $N$ & yes & yes & yes & yes & no \\
	guarantee depends on $d$ & yes & yes & no & no & no \\
	requires balancedness & yes & yes & no & no & yes \\
	proof technique & primal-dual & primal-dual & primal-dual & primal-dual & valid cuts \\
	access to cardinalities & no & no & no & no & yes \\
	outlier detection & no & no & yes & yes & yes \\  \hline
\end{tabular}
~\\~\\
\caption{Comparison of Recovery Guarantees for $K$-means Clustering Relaxations.}
\label{tab:comparison}
\end{table}

In this paper, we propose the first conic optimization scheme for the cardinality-constrained \mbox{$K$-means} clustering problem~\eqref{opt:kmean_size}. {\color{black} Our solution approach relies on an exact reformulation of problem~\eqref{opt:kmean_size} as an intractable mixed-integer linear program (MILP) to which we add a set of valid cuts before relaxing the resulting model to a tractable semidefinite program (SDP) or linear program (LP). The set of valid cuts is essential in strengthening these relaxations.} {\color{black} Both relaxations provide lower bounds on the optimal value of problem~\eqref{opt:kmean_size}, and they both recover the optimal value of~\eqref{opt:kmean_size} whenever a cluster separation condition is met. The latter requires all cluster diameters to be smaller than the distance between any two distinct clusters and, in case of outlier presence, also smaller than the distance between any outlier and any other point. The same condition (in the absence of outliers) was used in \cite{Elhamifar2012} and \cite{Awasthi2015}.} {\color{black} Our relaxations also give rise to deterministic rounding schemes which produce feasible solutions that are provably optimal in~\eqref{opt:kmean_size} whenever the cluster separation condition holds. Table~\ref{tab:comparison} compares our recovery guarantees to the results available in the literature. {\color{black}We emphasize that our guarantees are deterministic, that they apply to arbitrary data generating models, that they are dimension-independent, and that they hold for both our SDP and LP relaxations.} Finally, our algorithms extend to instances of~\eqref{opt:kmean_size} that are contaminated by outliers and whose cluster cardinalities $n_1, \ldots, n_K$ are not known precisely.} We summarize the paper's contributions as follows.
\begin{enumerate}
\item {\color{black}We derive a novel MILP reformulation of problem~\eqref{opt:kmean_size} that only involves $NK$ binary variables, as opposed to the standard MILP reformulation that contains $N^2$ binary variables, and whose LP relaxation is at least as tight as the LP relaxation of the standard reformulation.}
\item {\color{black} We develop lower bounds which exploit the cardinality information in problem~\eqref{opt:kmean_size}. Our bounds are tight whenever a cluster separation condition is met. Unlike similar results for other classes of clustering problems, our separation condition is deterministic, model-free and dimension-independent. Furthermore, our proof technique does not rely on the primal-dual argument of SDPs and LPs.}
\item We propose deterministic rounding schemes {\color{black} that transform the relaxed solutions to feasible solutions} for problem~\eqref{opt:kmean_size}. The solutions are optimal in~\eqref{opt:kmean_size} if the separation condition holds. To our best knowledge, we propose the first tractable solution scheme for problem~\eqref{opt:kmean_size} with optimality guarantees.
\item We illustrate that our lower bounds and rounding schemes extend to instances of problem~\eqref{opt:kmean_size} that are contaminated by outliers and whose cluster cardinalities are not known precisely.
\end{enumerate}

The remainder of the paper is structured as follows. Section 2 analyzes the cardinality-constrained $K$-means clustering problem~\eqref{opt:kmean_size} and derives the MILP reformulation underlying our solution scheme. Sections 3 and 4 propose and analyze our conic rounding approaches for problem~\eqref{opt:kmean_size} in the absence and presence of outliers, respectively. {\color{black}Section 5 presents numerical experiments, and Section~6 gives concluding remarks. Finally, a detailed description of the heuristic proposed by \cite{Bennett2000} for cardinality-constrained $K$-means clustering is provided in the appendix.}

\paragraph{Notation:} We denote by $\bm{1}$ the vector of all ones and by $\Vert \cdot \Vert$ the Euclidean norm. For symmetric square matrices $\mathbf{A},\mathbf{B}\in \mathbb S^N$, the relation $\mathbf{A} \succeq \mathbf{B}$ means that $\mathbf{A} - \mathbf{B}$ is positive semidefinite, while $\mathbf{A} \geq \mathbf{B}$ means that $\mathbf{A} - \mathbf{B}$ is elementwise non-negative. The notation $\left< \mathbf{A}, \mathbf{B} \right>=\text{Tr}(\mathbf{A}\mathbf{B})$ represents the trace inner product of $\mathbf{A}$ and $\mathbf{B}$. {\color{black}Furthermore, we use $\text{diag}(\mathbf{A})$ to denote a vector in $\mathbb{R}^N$ whose entries coincide with those of $\mathbf{A}$'s main diagonal. {\color{black}Finally, for a set of $N$ datapoints $\bm{\xi}_1,\ldots,\bm{\xi}_N$, we use $\mathbf{D} \in \mathbb{S}^N$ to denote the matrix of squared pairwise distances $d_{ij} = \Vert \bm\xi_i - \bm\xi_j \Vert^2$.}
\section{Problem Formulation and Analysis}
We first prove that the clustering problem~\eqref{opt:kmean_size} is an instance of a \emph{quadratic assignment problem} and transform~\eqref{opt:kmean_size} to an MILP with $NK$ binary variables. Then, we discuss the complexity of~\eqref{opt:kmean_size} and show that an optimal clustering always corresponds to some Voronoi partition of $\mathbb R^d$. 

{\color{black} Our first result relies on the following auxiliary lemma, which we state without proof.

\begin{lemma}
\label{lem:distance}
For any vectors $\bm\xi_1, \hdots, \bm\xi_n \in \mathbb{R}^d$, we have
\begin{equation*}
	\sum_{i=1}^n \Vert \bm\xi_i - {\textstyle \frac{1}{n} (\sum_{j=1}^n \bm\xi_j)} \Vert^2 = \frac{1}{2n} \sum_{i,j=1}^n \Vert \bm\xi_i - \bm\xi_j \Vert^2.
\end{equation*}
\end{lemma}

\begin{proof}{Proof}
See \citet[p.~1060]{Zha2001}. \qed
\end{proof}

}

{\color{black} Using Lemma~\ref{lem:distance}, \cite{Costa2017} notice that the $K$-means objective can be stated as a sum of quadratic terms. In the following proposition, we elaborate on this insight and prove that problem~\eqref{opt:kmean_size} is a specific instance of a quadratic assignment problem.}

\begin{proposition}[Quadratic Assignment Reformulation]
\label{prop:kmean_qap}
The clustering problem~\eqref{opt:kmean_size} can be cast as the quadratic assignment problem
\begin{equation}
\label{opt:kmean_qap}
\begin{aligned}
	& \textstyle \underset{\sigma \in \mathfrak S^N}{\text{\em{minimize}}} \quad \frac{1}{2} \left< \mathbf{W}, \mathbf{P}_{\sigma}^{} \mathbf{D} \mathbf{P}_{\sigma}^\top \right>,
\end{aligned}
\end{equation} 
where $\mathbf{W} \in \mathbb{S}^N$ is a block diagonal matrix with blocks $\frac{1}{n_k} \bm{11}^\top \in \mathbb{S}^{n_k}$, $k=1,\ldots,K$, $\mathfrak S^N$ is the set of permutations of $\{1,\ldots,N\}$, and $\mathbf{P}_{\sigma}$ is defined through $(\mathbf{P}_{\sigma})_{ij} = 1$ if $\sigma(i) = j$; ${\color{black}(\mathbf{P}_{\sigma})_{ij} = 0}$ otherwise.
\end{proposition}

\begin{proof}{Proof}
	We show that for any feasible solution of~\eqref{opt:kmean_size} there exists a feasible solution of~\eqref{opt:kmean_qap} which attains the same objective value and vice versa. To this end, for any partition $(I_1, \ldots, I_K)$ feasible in~\eqref{opt:kmean_size}, consider any permutation $\sigma \in \mathfrak{S}^N$ that satisfies $\sigma(\{ 1 + \sum_{i=1}^{k-1} n_i, \hdots, \sum_{i=1}^{k} n_i \}) = I_k$ for all $k = 1, \ldots, K${\color{black}, and denote its inverse by $\sigma^{-1}$}. This permutation is feasible in~\eqref{opt:kmean_qap}, and it achieves the same objective value as $(I_1, \ldots, I_K)$ in~\eqref{opt:kmean_size} because
\begin{align*}
	\sum_{k=1}^K \sum_{i \in I_k} \Vert \bm\xi_i - \tfrac{1}{n_k} ( \textstyle\sum_{j \in I_k} \bm\xi_j ) \Vert^2 &= \frac{1}{2} \sum_{k=1}^K \frac{1}{n_k} \sum_{i,j \in I_k} d_{ij} = \frac{1}{2} \sum_{k=1}^K \frac{1}{n_k} \sum_{i,j \in \sigma^{-1}(I_k)} d_{\sigma(i)\sigma(j)} = \frac{1}{2} \left< \mathbf{W}, \mathbf{P}_{\sigma} \mathbf{D} \mathbf{P}_{\sigma}^\top \right>,
\end{align*}
where the first equality is implied by Lemma~\ref{lem:distance}, the second equality is a consequence of the definition of $\sigma$, and the third equality follows from the definition of $\mathbf{W}$.

	Conversely, for any $\sigma \in \mathfrak{S}^N$ feasible in~\eqref{opt:kmean_qap}, consider any partition $(I_1, \ldots, I_K)$ satisfying $I_k = \sigma(\{ 1 + \sum_{i=1}^{k-1} n_i, \hdots, \sum_{i=1}^{k} n_i  \})$ for all $k = 1, \ldots, K$. This partition is feasible in~\eqref{opt:kmean_size}, and a similar reasoning as before shows that the partition achieves the same objective value as $\sigma$ in~\eqref{opt:kmean_qap}. \qed
\end{proof}
 
{\color{black} Generic quadratic assignment problems with $N$ facilities and $N$ locations can be reformulated as MILPs with $\Omega(N^2)$ binary variables via the Kaufmann and Broeckx linearization; see, {\em e.g.}, \citet[p.~2741]{Burkard2013}. The LP relaxations of these MILPs are, however, known to be weak, and give a trivial lower bound of zero; see, \emph{e.g.}, \citet[Theorem 4.1]{Zhang2013}. In Proposition~\ref{prop:kmean_milp} below we show that the intra-cluster permutation symmetry of the datapoints enables us to give an alternative MILP reformulation containing only $NK \ll \Omega(N^2)$ binary variables. We also mention that the related, yet different, cardinality-constrained exemplar-based clustering problem can be formulated as an MILP containing $\Omega(N^2)$ binary variables; see \cite{Mulvey1984}.}
 
\begin{proposition}[MILP Reformulation]
\label{prop:kmean_milp}
	The clustering problem~\eqref{opt:kmean_size} is equivalent to the MILP
\begin{equation}
\label{opt:kmean_milp} \tag{$\milp$}
\begin{aligned}
    &\text{~\em{minimize}} && \textstyle \frac{1}{2} \sum_{k=1}^K \frac{1}{n_k} \sum_{i,j=1}^N d_{ij}\eta_{ij}^k \\
    &\text{~\em{subject to}} && \pi_i^k \in \{ 0, 1 \},~ \eta_{ij}^k \in \mathbb{R}_+ && i,j = 1, \hdots, N,~ k = 1, \hdots, K\\
    &	  && \textstyle\sum_{i=1}^N \pi_i^k = n_k && k = 1, \hdots, K \\
    &	  && \textstyle\sum_{k=1}^K \pi_i^k = 1 && i = 1, \hdots, N \\
    &	  && \eta_{ij}^k \geq \pi_i^k + \pi_j^k - 1 && i,j = 1, \hdots, N,~ k = 1, \hdots, K .
\end{aligned}
\end{equation}
\end{proposition}

The binary variable $\pi_i^k$ in the MILP~\ref{opt:kmean_milp} satisfies $\pi_i^k=1$ if $i \in I_k$; ${\color{black}\pi_i^k=0}$ otherwise. At optimality, $\eta^k_{ij} = \max \{ \pi^k_i + \pi^k_j - 1, 0 \}$ is equal to 1 if $i,j \in I_k$ ({\em i.e.}, $\pi^k_i = \pi^k_j = 1$) and 0 otherwise.

\begin{proof}{Proof of Proposition~\ref{prop:kmean_milp}}
At optimality, the decision variables $\eta_{ij}^k$ in problem~\ref{opt:kmean_milp} take the values $\eta^k_{ij} = \max \{ \pi_i^k + \pi_j^k - 1, 0 \}$. Accordingly, problem~\ref{opt:kmean_milp} can equivalently be stated as
\begin{equation}
\tag{$\milp$$'$}
\label{opt:kmean_milp_prime}
\begin{aligned}
    &\text{~minimize} && \textstyle\frac{1}{2} \sum_{k=1}^K \frac{1}{n_k} \sum_{i,j=1}^N d_{ij}\max \{ \pi_i^k + \pi_j^k - 1, 0 \} \\
    &\text{~subject to} && \pi_i^k \in \{ 0, 1 \} &&\hspace{-30mm} i = 1, \hdots, N,~ k = 1, \hdots, K\\
    &      && \textstyle\sum_{i=1}^N \pi_i^k = n_k &&\hspace{-30mm} k = 1, \hdots, K \\
    &      && \textstyle\sum_{k=1}^K \pi_i^k = 1 &&\hspace{-30mm} i = 1, \hdots, N.
\end{aligned}
\end{equation}
In the following, we show that any feasible solution of~\eqref{opt:kmean_size} gives rise to a feasible solution of~\ref{opt:kmean_milp_prime} with the same objective value and vice versa. To this end, consider first a partition $(I_1, \ldots, I_K)$ that is feasible in~\eqref{opt:kmean_size}. Choosing $\pi_i^k = 1$ if $i \in I_k$ and $\pi_i^k = 0$ otherwise, $k = 1, \ldots, K$, is feasible in~\ref{opt:kmean_milp_prime} and attains the same objective value as $(I_1, \ldots, I_K)$ in~\eqref{opt:kmean_size} since
\begin{equation*}
	\sum_{k=1}^K \sum_{i \in I_k} \Vert \bm\xi_i - \tfrac{1}{n_k}( {\textstyle\sum_{j \in I_k} \bm\xi_j )} \Vert^2 = \frac{1}{2} \sum_{k=1}^K \frac{1}{n_k} \sum_{i,j \in I_k} d_{ij} = \frac{1}{2} \sum_{k=1}^K \frac{1}{n_k} \sum_{i,j=1}^N d_{ij}\max \{ \pi_i^k + \pi_j^k - 1, 0 \}.
\end{equation*}
Here, the first equality is implied by Lemma~\ref{lem:distance}, and the second equality follows from the construction of $\pi_i^k$. By the same argument, every $\pi_i^k$ feasible in~\ref{opt:kmean_milp_prime} gives rise to a partition $(I_1, \ldots, I_K)$, $I_k = \{ i : \pi_i^k = 1 \}$ for $k = 1, \ldots, K$, that is feasible in~\ref{opt:kmean_milp_prime} and that attains the same objective value. \qed
\end{proof}

{\color{black}
\begin{remark}
Note that zero is a (trivial) lower bound on the objective value of the LP relaxation of the MILP~\ref{opt:kmean_milp}. As a consequence, this LP relaxation is at least as tight as the LP relaxation of the Kaufmann and Broeckx exact MILP formulation of problem~\eqref{opt:kmean_qap}, which always yields the lower bound of zero. It is also possible to construct instances where the LP relaxation of the MILP~\ref{opt:kmean_milp} is strictly tighter.
\end{remark}
}

{\color{black}$K$-means clustering with cardinality constraints is known to be NP-hard as it is a special case of cardinality-constrained $p$-norm clustering, which was shown to be NP-hard (for any $p>1$) by \cite{Bertoni2012}. The restriction to the Euclidean norm (\emph{i.e.}, $p=2$), however, allows for a more concise proof, which is given in the following proposition.}

\begin{proposition}
\label{prop:np_hardness}
$K$-means clustering with cardinality constraints is NP-hard even for $K = 2$. Hence, unless {\em P $=$ NP}, there is no polynomial time algorithm for solving problem~\eqref{opt:kmean_size}.
\end{proposition}

\begin{proof}{Proof}
In analogy to Proposition~\ref{prop:kmean_milp}, one can show that the unconstrained $K$-means clustering problem can be formulated as a variant of problem \ref{opt:kmean_milp} that omits the first set of assignment constraints, which require that $\sum_{i=1}^N \pi^k_i = n_k$ for all k$ = 1, \ldots, K$, and replaces the (now unconstrained) cardinality $n_k$ in the objective function by the size of $I_k$, which can be expressed as $\sum_{i=1}^N \pi^k_i$. If $K = 2$, we can thus solve the unconstrained $K$-means clustering problem by solving problem \ref{opt:kmean_milp} for all cluster cardinality combinations $(n_1, n_2) \in \{ (1, N - 1), (2, N - 2), \ldots, (\lfloor N / 2 \rfloor, \lceil N / 2 \rceil) \}$ and selecting the clustering with the lowest objective value. {\color{black}Thus, in this case, if problem $\milp$ {\color{black}were} polynomial-time solvable, then so would be the unconstrained $K$-means clustering problem. This, however, would contradict} Theorem~1 in \cite{Aloise2009}, which shows that the unconstrained $K$-means clustering problem is NP-hard even for $K = 2$ clusters. \qed
\end{proof}

{\color{black} In the context of balanced clustering, similar hardness results have been established by \cite{Pyatkin2017}. Specifically, they prove that the balanced $K$-means clustering problem is NP-complete for $K \geq 2$ and $\frac{N}{K} \geq 3$  (\emph{i.e.}, the shared cardinality of all clusters is greater than or equal to three). In contrast, if $K \geq 2$ and $\frac{N}{K} = 2$ (\emph{i.e.}, each cluster should contain two points), balanced $K$-means clustering reduces to a minimum-weight perfect matching problem that can be solved in polynomial-time by different algorithms; see \citet[Table I]{Cook1999} for a review.}

In $K$-means clustering {\em without} cardinality constraints, the convex hulls of the optimal clusters do not overlap, and thus each cluster fits within a separate cell of a Voronoi partition of $\mathbb R^d$; see \emph{e.g.}, \citet[Theorem~2.1]{hasegawa1993cient}. We demonstrate below that this property is preserved in the presence of cardinality constraints.

\begin{theorem}[Voronoi Partition]
\label{thm:voronoi}
	For every optimal solution to problem~\eqref{opt:kmean_size}, there exists a Voronoi partition of $\mathbb{R}^d$ such that each cluster is contained in exactly one Voronoi cell.
\end{theorem}

\begin{proof}{Proof}
We show that for every optimal clustering $(I_1, \hdots, I_K)$ of \eqref{opt:kmean_size} and every $k, \ell \in \{ 1, \ldots, K \}$, $k<\ell$, there exists a hyperplane separating the points in $I_k$ from those in $I_{\ell}$. This in turn implies the existence of the desired Voronoi partition. 
{\color{black}Given a cluster $I_m$ for any $m \in \{1, \hdots, K \}$, define its cluster center as $\bm\zeta_m = \frac{1}{n_m} \sum_{i \in I_m} \bm\xi_i$, and let $\bm{h} = \bm\zeta_k - \bm\zeta_{\ell}$ be the vector that connects the cluster centers of $I_k$ and $I_\ell$.} The statement holds if $\bm{h}^\top (\bm\xi_{i_k} - \bm\xi_{i_{\ell}}) \geq 0$ for all $i_k \in I_k$ and $i_{\ell} \in I_{\ell}$ as $\bm{h}$ itself determines a separating hyperplane for $I_k$ and $I_{\ell}$ in that case. We thus assume that $\bm{h}^\top (\bm\xi_{i_k} - \bm\xi_{i_{\ell}}) < 0$ for some $i_k \in I_k$ and $i_{\ell} \in I_{\ell}$. However, this contradicts the optimality of the clustering $(I_1, \hdots, I_K)$~because
\begin{equation*}
\begin{aligned}
	\bm{h}^\top (\bm\xi_{i_k} - \bm\xi_{i_{\ell}}) < 0 \;
	&\Longleftrightarrow \; ( \bm\zeta_k - \bm\zeta_{\ell} )^\top (\bm\xi_{i_k} - \bm\xi_{i_{\ell}}) < 0 \\
	&\Longleftrightarrow \; \bm\xi_{i_k}^\top \bm\zeta_k +\ \bm\xi_{i_{\ell}}^\top \bm\zeta_{\ell} < \bm\xi_{i_k}^\top \bm\zeta_{\ell} + \bm\xi_{i_{\ell}}^\top \bm\zeta_k \\
	&\Longleftrightarrow \; \Vert \bm\xi_{i_{\ell}} - \bm\zeta_k \Vert^2 + \Vert \bm\xi_{i_k} - \bm\zeta_{\ell} \Vert^2 < \Vert \bm\xi_{i_k} - \bm\zeta_k \Vert^2 + \Vert \bm\xi_{i_{\ell}} - \bm\zeta_{\ell} \Vert^2,
\end{aligned}
\end{equation*}
where the last equivalence follows from multiplying both sides of the second inequality with $2$ and then completing the squares by adding $\bm\xi_{i_k}^\top \bm\xi_{i_k} + \bm\zeta_{k}^\top \bm\zeta_{k} + \bm\xi_{i_\ell}^\top \bm\xi_{i_\ell} + \bm\zeta_{\ell}^\top \bm\zeta_{\ell} $ on both sides.
Defining $\tilde{I}_k = I_k \cup \{ i_{\ell} \} \setminus \{ i_k \}$ and $\tilde{I}_{\ell}= I_{\ell} \cup \{ i_k \} \setminus \{ i_{\ell} \}$, the above would imply that
\begin{equation*}
\begin{aligned}
	&\sum_{i \in \tilde{I}_k} \Vert \bm\xi_i - \bm\zeta_k \Vert^2 + \sum_{i \in \tilde{I}_{\ell}} \Vert \bm\xi_i - \bm\zeta_{\ell} \Vert^2 + \sum_{\substack{m = 1,\ldots,K \\ m \not\in \{ k, \ell \}}} \; \sum_{i \in I_m} \Vert \bm\xi_i - \bm\zeta_m \Vert^2 \\
	& \quad < \ \sum_{i \in I_k} \Vert \bm\xi_i - \bm\zeta_k \Vert^2 + \sum_{i \in I_{\ell}} \Vert \bm\xi_i - \bm\zeta_{\ell} \Vert^2 + \sum_{\substack{m = 1,\ldots,K \\ m \not\in \{ k, \ell \}}} \; \sum_{i \in I_m} \Vert \bm\xi_i - \bm\zeta_m \Vert^2.
\end{aligned}
\end{equation*}
The left-hand side of the above inequality represents an upper bound on the sum of squared intra-cluster distances attained by the clustering $(I_1, \hdots, \tilde{I}_k, \ldots, \tilde{I}_{\ell}, \ldots, I_K)$ since $\bm\zeta_k$ and $\bm\zeta_{\ell}$ may not coincide with the minimizers $ \tfrac{1}{n_k}\sum_{i \in \tilde{I}_k} \bm\xi_i$ and $\tfrac{1}{n_\ell}\sum_{i \in \tilde{I}_{\ell}} \bm\xi_i$, respectively. {\color{black}Recall that the cluster centers are chosen so as to minimize the sum of the distances from the cluster center to each point in the cluster.} We thus conclude that the clustering $(I_1, \hdots, \tilde{I}_k, \ldots, \tilde{I}_{\ell}, \ldots, I_K)$ attains a strictly lower objective value than $(I_1, \hdots, I_K)$ in problem~\eqref{opt:kmean_size}, which is a contradiction. \qed
\end{proof}

\section{Cardinality-Constrained Clustering without Outliers}
\label{sec:no-outliers}
We now relax the intractable MILP~\ref{opt:kmean_milp} to tractable conic programs that yield efficiently computable lower and upper bounds on~\ref{opt:kmean_milp}.

\subsection{Convex Relaxations and Rounding Algorithm} 
\label{sec:relax}
We first eliminate the $\eta^k_{ij}$ variables from~\ref{opt:kmean_milp} by re-expressing the problem's objective function as
\begin{equation*}
\begin{aligned}	
	\frac{1}{2}\sum_{k=1}^K \frac{1}{n_k} \sum_{i,j=1}^N d_{ij}\eta^k_{ij}
	= \frac{1}{2}\sum_{k=1}^K \frac{1}{n_k} \sum_{i,j=1}^N d_{ij} \max \{ \pi^k_i + \pi^k_j - 1, 0 \}
	= \frac{1}{2}\sum_{k=1}^K \frac{1}{n_k} \sum_{i,j=1}^N d_{ij}\pi^k_{i} \pi^k_j,
\end{aligned}
\end{equation*}
where the last equality holds because the variables $\pi^k_i$ are binary. Next, we apply the variable transformation $x^k_i \leftarrow 2\pi^k_i - 1$, whereby \ref{opt:kmean_milp} simplifies to
\begin{equation}
\begin{aligned}
	\label{eq:kmean_obj2}
    &\text{minimize} && \textstyle \frac{1}{8} \sum_{k=1}^K \frac{1}{n_k} \sum_{i,j=1}^N d_{ij}(1 + x^k_i)(1 + x^k_j) \\
    &\text{subject to} && x^k_i \in \{ -1, +1\} &&\hspace{-15mm} i= 1, \hdots, N,~ k = 1, \hdots, K\\
    &	  && \textstyle \sum_{i=1}^N x^k_i = 2n_k - N &&\hspace{-15mm} k = 1, \hdots, K \\
    &	  && \textstyle \sum_{k=1}^K x^k_i = 2 - K &&\hspace{-15mm} i = 1, \hdots, N.
\end{aligned}
\end{equation}
Here, $x^k_i$ takes the value $+1$ if the $i$-th datapoint is assigned to cluster $k$ and $-1$ otherwise. Note that the constraints in~\eqref{eq:kmean_obj2} are indeed equivalent to the first two constraints in~\ref{opt:kmean_milp}, respectively. 
In Theorem~\ref{thm:kmean_sdp} below we will show that the reformulation~\eqref{eq:kmean_obj2} of the MILP~\ref{opt:kmean_milp} admits the SDP relaxation
\begin{equation}
\label{opt:kmean_sdp} \tag{$\mathcal{R}_{\text{SDP}}$}
\begin{aligned}
	&~\text{{minimize}} && \textstyle \frac{1}{8} \left< \mathbf{D}, \sum_{k=1}^K \frac{1}{n_k} \left(  \mathbf{M}^k +\bm{1} \bm{1}^\top + \bm x^k\bm{1}^\top + \bm{1} (\bm x^k)^\top \right) \right> \\
    &\text{~{subject to}} && (\bm{x}^k, \mathbf{M}^k) \in \mathcal{C}^{}_{\text{SDP}} (n_k) \quad k = 1, \hdots, K \\
    & && \textstyle\sum_{k=1}^K \bm{x}^k = ( 2 - K ) \bm{1},
\end{aligned}
\end{equation}
where, for any $n\in\mathbb N$, the convex set $\mathcal{C}^{}_{\text{SDP}} (n) \subset \mathbb{R}^N \times \mathbb{S}^N$ is defined as
\begin{equation*}
	\mathcal{C}^{}_{\text{SDP}} (n) = \left\{ (\bm{x}, \mathbf{M}) \in \mathbb{R}^N \times \mathbb{S}^N:~ 
		\begin{array}{l}		
		\bm{1}^\top \bm{x} = 2n - N, \; \mathbf{M} \bm{1} = (2n-N) \bm{x} \\
		\text{diag}(\mathbf{M}) = \bm{1}, \; \mathbf{M} \succeq \bm{xx}^\top \\
		\mathbf{M} + \bm{1}\bm{1}^\top + \bm{x}\bm{1}^\top + \bm{1}\bm{x}^\top \geq \bm{0} \\
		\mathbf{M} + \bm{1}\bm{1}^\top - \bm{x}\bm{1}^\top - \bm{1}\bm{x}^\top \geq \bm{0} \\
		\mathbf{M} - \bm{1}\bm{1}^\top + \bm{x}\bm{1}^\top - \bm{1}\bm{x}^\top \leq \bm{0} \\
		\mathbf{M} - \bm{1}\bm{1}^\top - \bm{x}\bm{1}^\top + \bm{1}\bm{x}^\top \leq \bm{0}
		\end{array}
	\right\}.
\end{equation*}
Note that $\mathcal{C}^{}_{\text{SDP}} (n)$ is semidefinite representable because Schur's complement allows us to express the constraint $\mathbf{M} \succeq \bm{xx}^\top$ as a linear matrix inequality; see, {\em e.g.}, \cite{Boyd2004}. Furthermore, we point out that the last four constraints in $\mathcal{C}^{}_{\text{SDP}} (n)$ are also used in the \emph{reformulation-linearization technique} for nonconvex programs, as described by \cite{Anstreicher2009}.

We can further relax the above SDP to an LP, henceforth denoted by~$\rlp$, where the constraints $(\bm{x}^k, \mathbf{M}^k) \in \mathcal{C}^{}_{\text{SDP}} (n_k)$ are replaced with $(\bm{x}^k, \mathbf{M}^k) \in \mathcal{C}^{}_{\text{LP}} (n_k)$, and where, for any $n\in\mathbb N$, the polytope $\mathcal{C}^{}_{\text{LP}} (n)$ is obtained by removing the non-linear constraint $\mathbf{M} \succeq \bm{xx}^\top$ from $\mathcal{C}^{}_{\text{SDP}} (n)$.

\begin{theorem}[SDP and LP Relaxations]
\label{thm:kmean_sdp}
	We have {\em $\min \rlp$ $\leq \min$~\ref{opt:kmean_sdp} $\leq \min$~\ref{opt:kmean_milp}}.
\end{theorem}

\begin{proof}{Proof}
	The inequality $\min \rlp \leq \min \rsdp$ is trivially satisfied because $\mathcal{C}^{}_\text{SDP}(n)$ is constructed as a subset of $\mathcal{C}^{}_\text{LP}(n)$ for every $n\in\mathbb N$. To prove the inequality $\min \rsdp \leq \min \milp$, consider any set of binary vectors $\{\bm{x}^k \}_{k=1}^K$ feasible in~\eqref{eq:kmean_obj2} and define $\mathbf{M}^k = \bm{x}^k (\bm{x}^k)^\top$ for $k=1,\ldots,K$. By construction,  the objective value of $\{\bm{x}^k \}_{k=1}^K$ in~\eqref{eq:kmean_obj2} coincides with that of $\{(\bm x^k,\mathbf M^k)\}_{k=1}^K$ in~\ref{opt:kmean_sdp}. Moreover, the constraints in~\eqref{eq:kmean_obj2} imply that
\begin{equation*}
	\mathbf{M}^k \bm{1} = \bm{x}^k (\bm{x}^k)^\top\bm{1} = (2n_k-N) \bm x^k, \quad  
	\text{{diag}}(\mathbf{M}^k) = \bm{1}, \quad
	\mathbf{M}^k \succeq \bm x^k (\bm x^k)^\top
\end{equation*}
and
\begin{equation*}
\begin{aligned}
	&\mathbf{M}^k + \bm{11}^\top + \bm{x}^k\bm{1}^\top + \bm{1}(\bm{x}^k)^\top = +(\bm{1} + \bm{x}^k)(\bm{1} + \bm{x}^k)^\top \geq \bm{0} \\
	&\mathbf{M}^k + \bm{11}^\top - \bm{x}^k\bm{1}^\top - \bm{1}(\bm{x}^k)^\top = +(\bm{1} - \bm{x}^k)(\bm{1} - \bm{x}^k)^\top \geq \bm{0} \\
	&\mathbf{M}^k - \bm{11}^\top + \bm{x}^k\bm{1}^\top - \bm{1}(\bm{x}^k)^\top = -(\bm{1} - \bm{x}^k)(\bm{1} + \bm{x}^k)^\top \leq \bm{0} \\
	&\mathbf{M}^k - \bm{11}^\top - \bm{x}^k\bm{1}^\top + \bm{1}(\bm{x}^k)^\top = -(\bm{1} + \bm{x}^k)(\bm{1} - \bm{x}^k)^\top \leq \bm{0},
\end{aligned}
\end{equation*}
which ensures that $(\bm{x}^k, \mathbf{M}^k) \in \mathcal{C}^{}_\text{SDP}(n_k)$ for every $k$. Finally, the constraint $\sum_{k=1}^K \bm x^k=(2-K)\bm 1$ in~\ref{opt:kmean_sdp} coincides with the last constraint in~\eqref{eq:kmean_obj2}. Thus, $\{(\bm x^k,\mathbf M^k)\}_{k=1}^K$ is feasible in~\ref{opt:kmean_sdp}. The desired inequality now follows because any feasible point in~\eqref{eq:kmean_obj2} corresponds to a feasible point in~\ref{opt:kmean_sdp} with the same objective value. Note that the converse implication is generally false. \qed
\end{proof}

\begin{remark}
In the special case when $K = 2$, we can half the number of variables in {\em \ref{opt:kmean_sdp}} and $\rlp$ by setting $\bm{x}^2=-\bm{x}^1$ and $\mathbf{M}^2=\mathbf{M}^1$ without loss of generality.
\end{remark}

{\color{black}It is possible to show that $\mathcal{R}_{\rm{LP}}$ is at least as tight as the na\"ive LP relaxation $\mathcal{L}$ of the MILP~\ref{opt:kmean_milp}, where the integrality constraints are simply ignored. One can also construct instances where $\mathcal{R}_{\rm{LP}}$ is strictly tighter than $\mathcal{L}$. We also emphasize that both LP relaxations entail $\mathcal{O}(N^2K)$ variables and $\mathcal{O}(N^2K)$ constraints.

\begin{proposition}
\label{prop:rlp_vs_milplp}
We have $\min \mathcal{R}_{\rm{LP}}$ $\geq \min \mathcal{L}$.
\end{proposition}

\begin{proof}{Proof}
Consider a feasible solution $\{ (\bm{x}^k, \mathbf{M}^k) \}_{k=1}^K$ of $\mathcal{R}_{\text{LP}}$. Its feasibility implies that 
$$\text{(a) } \textstyle\sum_{k = 1}^K x^k_i = 2 - K~\forall i, \qquad
    \text{(b) } \textstyle\sum_{i = 1}^N x^k_i = 2n_k - N~\forall k, \qquad 
    \text{(c) }  m_{ij}^k - x^k_i - x^k_j + 1 \geq 0~\forall i,j,k.
$$ 
Next, set $\pi^k_i = (x^k_i + 1)/2$ and $\eta^k_{ij} = \frac{1}{4}(m^k_{ij} + x^k_i + x^k_j + 1)$ for all $i,j,k$. Then,
$$\text{(a') } \textstyle\sum_{k = 1}^K \pi^k_i = 1~\forall i, \qquad 
    \text{(b') } \textstyle\sum_{i = 1}^N \pi^k_i = n_k~\forall k, \qquad 
    \text{(c') } \eta^k_{ij} \geq \pi^k_i + \pi^k_j - 1~\forall i,j,k.
$$ 
Hence, this solution is feasible in $\mathcal{L}$. A direct calculation also reveals that both solutions attain the same objective value in their respective optimization problems. This confirms that $\mathcal{R}_{\text{LP}}$ is a relaxation that is at least as tight as $\mathcal{L}$. \qed
\end{proof}
}

Next, we develop a rounding algorithm that recovers a feasible clustering (and thus an upper bound on~\ref{opt:kmean_milp}) from an optimal solution of the relaxed problem~\ref{opt:kmean_sdp} or $\rlp$; see Algorithm~\ref{alg:det_rounding_multi}.
 
\begin{algorithm}
\caption{Rounding algorithm for cardinality-constrained clustering}
\label{alg:det_rounding_multi}
\begin{algorithmic}[1]
\State \textbf{Input:} $\mathcal{I}_1=\{1,\ldots,N\}$ (data indices), $n_k\in\mathbb N, \ k = 1,\ldots,K$ (cluster sizes).
\State Solve~\ref{opt:kmean_sdp} or $\rlp$ for the datapoints $\bm\xi_i$, $i \in \mathcal{I}_1$, and record the optimal $\bm x^1,\ldots, \bm x^K \in\mathbb R^N$.
\State Solve the linear assignment problem
\begin{equation*}
	\mathbf{\Pi}' \in
	\underset{\mathbf{\Pi}}{\text{argmax}} \left\{ 
		\sum_{i = 1}^N \sum_{k = 1}^K \pi_i^kx^k_i:~
		\pi_i^k \in \{ 0, 1\},~
		\sum_{i = 1}^N \pi_i^k = n_k \;\; \forall k,~
		\sum_{k = 1}^K \pi_i^k = 1 \;\; \forall i
	\right\}.
\end{equation*}
\State Set $I'_k \leftarrow \{ i: (\pi')_i^k = 1 \}$ for all $k = 1, \hdots, K$.
\State Set $\bm\zeta_k \leftarrow \frac{1}{n_k}\sum_{i \in I'_k} \bm\xi_i$ for all $k = 1, \hdots, K$.
\State Solve the linear assignment problem
\begin{equation*}
	\mathbf{\Pi}^\star \in
	\underset{\mathbf{\Pi}}{\text{argmin}} \left\{ 
		\sum_{i = 1}^N \sum_{k = 1}^K \pi_i^k \Vert \bm\xi_i - \bm\zeta_k \Vert^2 :~
		\pi_i^k \in \{ 0, 1\},~
		\sum_{i = 1}^N \pi_i^k = n_k \;\; \forall k,~
		\sum_{k = 1}^K \pi_i^k = 1 \;\; \forall i
	\right\}.
\end{equation*}
\State Set $I_k \leftarrow \{ i: (\pi^\star)_i^k = 1 \}$ for all $k = 1, \hdots, K$.
\State \textbf{Output:} $I_1, \hdots, I_K$.
\end{algorithmic}
\end{algorithm}

Recall that the continuous variables $\bm x^k=(x_1^k,\hdots,x^k_N)^\top$ in~\ref{opt:kmean_sdp} and $\rlp$ correspond to the binary variables in~\eqref{eq:kmean_obj2} with identical names. This correspondence motivates us to solve a linear assignment problem in Step~3 of Algorithm~\ref{alg:det_rounding_multi}, which seeks a matrix $\bm \Pi\in\{0,1\}^{N\times K}$ with $\pi_i^k\approx \frac{1}{2}(x_i^k+1)$ for all $i$ and $k$ subject to the prescribed cardinality constraints. Note that even though this assignment problem constitutes an MILP, it can be solved in polynomial time because its constraint matrix is totally unimodular, implying that its LP relaxation is exact. Alternatively, one may solve the assignment problem using the Hungarian algorithm; see, {\em e.g.},~\cite{Burkard2009}.
 
Note that Steps 5--7 of Algorithm~\ref{alg:det_rounding_multi} are reminiscent of a \emph{single} iteration of Lloyd's algorithm~for cardinality-constrained $K$-means clustering as described by \cite{Bennett2000}. Specifically, Step~5 calculates the cluster centers $\bm\zeta_k$, while Steps~6 and~7 reassign each point to the nearest~center while adhering to the cardinality constraints. Algorithm~\ref{alg:det_rounding_multi} thus follows just one step of Lloyd's algorithm initialized with an optimizer of~\ref{opt:kmean_sdp} or~$\rlp$. This refinement step ensures that the output clustering is compatible with a Voronoi partition of $\mathbb R^d$, which is desirable~in~view~of~Theorem~\ref{thm:voronoi}. 

\subsection{Tighter Relaxations for Balanced Clustering}
\label{sec:balanced_clustering}

The computational burden of solving~\ref{opt:kmean_sdp} and $\rlp$ grows with $K$. We show in this section that if all clusters share the same size $n$ ({\em i.e.}, $n_k = n$ for all $k$), then~\ref{opt:kmean_sdp} can be replaced by
\begin{equation}
\label{opt:kmean_sdp_balanced} \tag{$\mathcal{R}_{\text{SDP}}^{\text{b}}$}
\begin{aligned}	
	\hspace{-3mm}&\text{minimize} \hspace{-1mm}&& \textstyle  \frac{1}{8n} \left< \mathbf{D}, \mathbf{M}^1 + \bm{11}^\top \!\!+ \bm{x}^1\bm{1}^\top\!\! + \bm{1}(\bm{x}^1)^\top\!\! + (K-1) \left( \mathbf{M} + \bm{11}^\top\!\! + \bm{x1}^\top\!\! + \bm{1x}^\top \right) \right> \hspace{-3mm} \\
	\hspace{-3mm}&\text{subject to} \hspace{-1mm}&& (\bm{x}^1, \mathbf{M}^1), (\bm{x}, \mathbf{M}) \in \mathcal{C}^{}_{\text{SDP}}(n), \quad \bm{x}^1 + (K-1)\bm{x} = (2-K)\bm{1}, \quad x^1_1 = 1,
\end{aligned}
\end{equation}
whose size no longer scales with $K$. Similarly,~$\rlp$ simplifies to the LP~$\rlpb$ obtained from \ref{opt:kmean_sdp_balanced} by replacing $\mathcal{C}^{}_{\text{SDP}}(n)$ with $\mathcal{C}^{}_{\text{LP}}(n)$. {\color{black} This is a manifestation of how symmetry can be exploited to simplify convex programs, a phenomenon which is studied in a more general setting by \cite{GatermannParrilo2004}.}

\begin{corollary}[Relaxations for Balanced Clustering]
\label{cor:kmean_sdp_balanced}
We have~{\em $\min \rlpb$ $\leq \min$~\ref{opt:kmean_sdp_balanced} $\leq \min$~\ref{opt:kmean_milp}}.
\end{corollary}

\begin{proof}{Proof}
	The inequality $\min \rlpb \leq \min$~\ref{opt:kmean_sdp_balanced} is trivially satisfied. To prove the inequality $\min$~\ref{opt:kmean_sdp_balanced} $\leq \min$~\ref{opt:kmean_milp}, we first add the symmetry breaking constraint $x_1^1 = 1$ to the MILP~\ref{opt:kmean_milp}. Note that this constraint does not increase the optimal value of~\ref{opt:kmean_milp}. It just requires that the cluster containing the datapoint $\bm \xi_1$ should be assigned the number $k=1$. This choice is unrestrictive because all clusters have the same size. By repeating the reasoning that led to Theorem~\ref{thm:kmean_sdp}, the MILP~\ref{opt:kmean_milp} can then be relaxed to a variant of the SDP~\ref{opt:kmean_sdp} that includes the (linear) symmetry breaking constraint $x_1^1 = 1$. Note that the constraints and the objective function of the resulting SDP are invariant under permutations of the cluster indices $k=2,\hdots,K$ because $n_k=n$ for all $k$. Note also that the constraints are not invariant under permutations involving $k=1$ due to the symmetry breaking constraint. Next, consider any feasible solution $\{(\bm{x}^k, \mathbf{M}^k)\}_{k=1}^K$ of this SDP, and define
\begin{equation*}
	\bm{x} = \frac{1}{K-1}\sum_{k=2}^K \bm{x}^k \quad \text{and} \quad
	\mathbf{M} = \frac{1}{K-1}\sum_{k=2}^K \mathbf{M}^k.
\end{equation*}
Moreover, construct a permutation-symmetric solution $\{(\bm{x}_{\rm s}^k, \mathbf{M}_{\rm s}^k)\}_{k=1}^K$ by setting 
\begin{equation*}
\begin{aligned}
	&\bm{x}_{\rm s}^1=\bm x^1, && \bm{x}_{\rm s}^k=\bm x && \forall k=2,\ldots,K,\\
	&\mathbf{M}_{\rm s}^1=\mathbf{M}^1, && \mathbf{M}_{\rm s}^k=\mathbf{M} && \forall k=2,\ldots,K.
\end{aligned}
\end{equation*}
By the convexity and permutation symmetry of the SDP, the symmetrized solution $\{(\bm{x}_{\rm s}^k, \mathbf{M}_{\rm s}^k)\}_{k=1}^K$ is also feasible in the SDP and attains the same objective value as $\{(\bm{x}^k, \mathbf{M}^k)\}_{k=1}^K$. Moreover, as the choice of $\{(\bm{x}^k, \mathbf{M}^k)\}_{k=1}^K$ was arbitrary, we may indeed restrict attention to symmetrized solutions with $\bm x^k=\bm x^\ell$ and $\mathbf M^k=\mathbf M^\ell$ for all $k,\ell\in\{2,\hdots, K\}$ without increasing the objective value of the SDP. Therefore, the simplified SDP relaxation~\ref{opt:kmean_sdp_balanced} provides a lower bound on~\ref{opt:kmean_milp}. \qed
\end{proof}

If $n_k=n$ for all $k$, then the SDP and LP relaxations from Section~\ref{sec:relax} admit an optimal solution where both $\bm{x}^k$ and $\mathbf{M}^k$ are independent of $k$, in which case Algorithm~\ref{alg:det_rounding_multi} performs poorly. This motivates the improved relaxations \ref{opt:kmean_sdp_balanced} and $\rlpb$ involving the symmetry breaking constraint $x^1_1 = 1$, which ensures that---without loss of generality---the cluster harboring the first datapoint $\bm\xi_1$ is indexed by $k=1$. As the symmetry between clusters $2,\hdots,K$ persists and because any additional symmetry breaking constraint would be restrictive, the optimal solutions of \ref{opt:kmean_sdp_balanced} and $\rlpb$ only facilitate a reliable recovery of cluster~1. To recover {\em all} clusters, however, we can solve \ref{opt:kmean_sdp_balanced} or $\rlpb$ $K-1$ times over the yet unassigned datapoints, see Algorithm~\ref{alg:det_rounding_balanced}. The resulting clustering could be improved by appending one iteration of Lloyd's algorithm (akin to Steps 5--7 in Algorithm~\ref{alg:det_rounding_multi}).

{\color{black}
In contrast, the na\"ive relaxation $\mathcal{L}$ of $\milp$ becomes significantly weaker when all cardinalities are equal. To see this, we note that a solution $\pi^k_i = 1/K$ and $\eta^k_{ij} = 0$ for all $i,j = 1, \hdots, N$ and for all $k = 1, \hdots, K$ is feasible in $\mathcal{L}$ (\emph{i.e.}, it satisfies all constraints in problem $\milp$ except the integrality constraints which are imposed on $\pi^k_i$) whenever $K \geq 2$. Hence, the optimal objective value of $\mathcal{L}$ is zero. This could be avoided by adding a symmetry breaking constraint $\pi^1_1 = 1$ to problem $\mathcal{L}$ to ensure that the cluster containing the first datapoint $\bm\xi_1$ is indexed by $k = 1$. However, the improvement appears to be marginal.
}

\begin{algorithm}
\caption{Rounding algorithm for balanced clustering}
\label{alg:det_rounding_balanced}
\begin{algorithmic}[1]
\State \textbf{Input:} $\mathcal{I}_1=\{1,\ldots,N\}$ (data indices), $n\in\mathbb N$ (cluster size), $K=N/n\in\mathbb N$ (\# clusters).
\For{$k =1,\hdots, K-1$}
	\State Solve~\ref{opt:kmean_sdp_balanced} or $\rlpb$ for the datapoints $\bm\xi_i$, $i \in \mathcal{I}_k$, and record the optimal $\bm x^1\in\mathbb R^{\vert \mathcal{I}_k \vert}$.
	\State Determine a bijection $\rho : \{1, \ldots,\vert \mathcal{I}_k \vert \} \rightarrow \mathcal{I}_k$ such that $x_{\rho(1)}^1 \geq x_{\rho(2)}^1\geq\cdots\geq x_{\rho({\vert \mathcal{I}_k \vert})}^1$.
	\State Set $I_k \leftarrow \{\rho(1), \ldots, \rho(n)\}$ and $\mathcal{I}_{k+1} \leftarrow \mathcal{I}_k \backslash I_k$.
\EndFor
\State Set $I_K \leftarrow \mathcal{I}_K$.
\State \textbf{Output:} $I_1, \hdots, I_K$.
\end{algorithmic}
\end{algorithm}

{\color{black}
\subsection{Comparison to existing SDP Relaxations}

We now compare $\rsdp$ and $\rsdp^{\rm{b}}$ with existing SDP relaxations from the literature. First, we report the various SDP relaxations proposed by \cite{Peng2007} and \cite{Awasthi2015}. Then, we establish that two of them are equivalent. Finally, we show that $\rsdp$ and $\rsdp^{\rm{b}}$ are relaxations that are at least as tight as their corresponding counterparts from the literature. The numerical experiments in Section~\ref{sec:numericalexp} provide evidence that this relation can also be strict.

\cite{Peng2007} suggest two different SDP relaxations for the \emph{un}constrained $K$-means clustering problem and an SDP relaxation for the balanced $K$-means clustering problem. All of them involve a Gram matrix $\mathbf{W} \in \mathbb{S}^N$ with entries $w_{ij} = \bm{\xi}_i^\top \bm{\xi}_j$. Their stronger relaxation for the unconstrained $K$-means clustering problem takes the form
\begin{equation}
\label{opt:sdp_pw1} \tag{$\mathcal{PW}_{1}$}
\begin{aligned}
    &\text{minimize} && \hspace{-1mm}\left< \mathbf{W}, \mathbb{I}-\mathbf{Z} \right> \\
    &\text{subject to} && \mathbf{Z} \in \mathbb{S}^N \\
    &	&&\mathbf{Z} \succeq \bm 0,~\mathbf{Z} \geq \bm 0,~\mathbf{Z}\bm{1} = \bm{1},~\text{Tr}(\mathbf{Z})=K,
\end{aligned}
\end{equation}
where $\mathbb{I}$ denotes the identity matrix of dimension $N$. Note that the constraints $\mathbf{Z} \geq \mathbf{0}$ and $\mathbf{Z}\bm{1} = \bm{1}$ ensure that $\mathbf{Z}$ is a stochastic matrix, and hence all of its eigenvalues lie between 0 and 1. Thus, further relaxing the non-negativity constraints leads to the following weaker relaxation,
\begin{equation}
\label{opt:sdp_pw2} \tag{$\mathcal{PW}_{2}$}
\begin{aligned}
    &\text{minimize} && \hspace{-1mm}\left< \mathbf{W}, \mathbb{I}-\mathbf{Z} \right> \\
    &\text{subject to} && \mathbf{Z} \in \mathbb{S}^N \\
    &	&&\mathbb{I} \, \succeq \, \mathbf{Z} \, \succeq \, \bm 0,~\mathbf{Z}\bm{1} = \bm{1},~\text{Tr}(\mathbf{Z})=K.
\end{aligned}
\end{equation}
\cite{Peng2007} also demonstrate that~\ref{opt:sdp_pw2} essentially reduces to an eigenvalue problem, which implies that one can solve \ref{opt:sdp_pw2} in $\mathcal{O}(KN^2)$ time; see \cite{GolubLoan1996}. Their SDP relaxation for the \emph{balanced} $K$-means clustering problem is similar to \ref{opt:sdp_pw1} and takes the form
\begin{equation}
\label{opt:sdp_pw3} \tag{$\mathcal{PW}_{1}^{\rm{b}}$}
\begin{aligned}
    &\text{minimize} && \hspace{-1mm}\left< \mathbf{W}, \mathbb{I}-\mathbf{Z} \right> \\
    &\text{subject to} && \mathbf{Z} \in \mathbb{S}^N \\
    &	&&\mathbf{Z} \succeq \bm 0,~\bm 0 \leq  \mathbf{Z} \leq (K/N)\bm{11}^\top ,~\mathbf{Z}\bm{1} = \bm{1},~\text{Tr}(\mathbf{Z})=K.
\end{aligned}
\end{equation}

\cite{Awasthi2015} suggest another SDP relaxation for the unconstrained $K$-means clustering problem, based on the same matrix of squared pairwise distances $\mathbf{D}$ considered in this paper,
\begin{equation}
\label{opt:sdp_awasthi} \tag{$\mathcal{A}$}
\begin{aligned}
    &\text{minimize} && \left< \mathbf{D}, \mathbf{Z} \right> \\
    &\text{subject to} && \mathbf{Z} \in \mathbb{S}^N \\
    &	&&\mathbf{Z} \succeq \bm 0,~\mathbf{Z} \geq \bm 0,~\mathbf{Z}\bm{1} = \bm{1},~\text{Tr}(\mathbf{Z})=K.
\end{aligned}
\end{equation}

The following observation asserts that the stronger relaxation \ref{opt:sdp_pw1} of \cite{Peng2007} and the relaxation \ref{opt:sdp_awasthi} of \cite{Awasthi2015} are actually equivalent.

\begin{observation}
\label{obs:equi_pw_aw}
	{\em The problems \ref{opt:sdp_pw1} and \ref{opt:sdp_awasthi} are equivalent.}
\end{observation}

\begin{proof}{Proof}
	Begin by expressing the objective of \ref{opt:sdp_pw1} in terms of the pairwise distance matrix~$\mathbf{D}$,
\begin{equation}
\label{eq:equiv_gram_dist}
\begin{aligned}
	\big\langle \mathbf{W}, \mathbb{I}-\mathbf{Z} \big\rangle &= \frac{1}{2} \Big[ 2 \big\langle \mathbf{W}, \mathbb{I} \big\rangle - \big\langle 2\mathbf{W}, \mathbf{Z} \big\rangle - \big\langle \mathbf{D}, \mathbf{Z} \big\rangle \Big] + \frac{1}{2} \big\langle \mathbf{D}, \mathbf{Z} \big\rangle \\[2mm]
	&=\, \frac{1}{2} \Big[ 2 \big\langle \mathbf{W}, \mathbb{I} \big\rangle - \big\langle 2\mathbf{W}+\mathbf{D}, \mathbf{Z} \big\rangle \Big] + \frac{1}{2} \big\langle \mathbf{D}, \mathbf{Z} \big\rangle \\[2mm]
	&\stackrel{\text{(a)}}{=}\, \frac{1}{2} \Big[ 2 \big\langle \mathbf{W}, \mathbb{I} \big\rangle - \big\langle \bm{1} \, \text{diag}(\mathbf{W})^\top+\text{diag}(\mathbf{W}) \, \bm{1}^\top, \mathbf{Z} \big\rangle \Big] + \frac{1}{2} \big\langle \mathbf{D}, \mathbf{Z} \big\rangle \\[2mm]
	&=\, \frac{1}{2} \Big[ 2 \big\langle \mathbf{W}, \mathbb{I} \big\rangle - \big\langle \bm{1} \, \text{diag}(\mathbf{W})^\top, \mathbf{Z} \big\rangle - \big\langle \text{diag}(\mathbf{W}) \, \bm{1}^\top, \mathbf{Z} \big\rangle \Big] + \frac{1}{2} \big\langle \mathbf{D}, \mathbf{Z} \big\rangle \\[2mm]
	&=\, \frac{1}{2} \Big[ 2\text{Tr} \big( \mathbf{W} \big) - \text{Tr} \big( \mathbf{Z} \, \bm{1} \, \text{diag}(\mathbf{W})^\top \big) - \text{Tr} \big( \text{diag}(\mathbf{W}) \, \bm{1}^\top \, \mathbf{Z} \big) \Big] + \frac{1}{2} \big\langle \mathbf{D}, \mathbf{Z} \big\rangle \\[2mm]
	&\stackrel{\text{(b)}}{=}\, \frac{1}{2} \Big[ 2\text{Tr} \big( \mathbf{W} \big) - \text{Tr} \big( \bm{1} \, \text{diag}(\mathbf{W})^\top \big) - \text{Tr} \big( \text{diag}(\mathbf{W}) \, \bm{1}^\top \big) \Big] + \frac{1}{2} \big\langle \mathbf{D}, \mathbf{Z} \big\rangle \\[2mm]
	&=\, \frac{1}{2} \Big[ 2 \, \big( \bm{1}^\top \text{diag}(\mathbf{W}) \big) - \bm{1}^\top \text{diag}(\mathbf{W}) - \bm{1}^\top \text{diag}(\mathbf{W}) \Big] + \frac{1}{2} \big\langle \mathbf{D}, \mathbf{Z} \big\rangle \\[2mm]
	&=\, \frac{1}{2} \big\langle \mathbf{D}, \mathbf{Z} \big\rangle. 
\end{aligned}
\end{equation}
Here, (a) follows from the observation that the $ij$-th element of the matrix $2\mathbf{W}+\mathbf{D}$ can be written as $2 \, \bm{\xi}_i^\top\bm{\xi}_j + \Vert \bm{\xi}_i-\bm{\xi}_j \Vert^2 = \Vert \bm{\xi}_i \Vert^2 + \Vert \bm{\xi}_j \Vert^2$, and (b) uses the insights that $\mathbf{Z}\bm{1} = \bm{1}$ and $\bm{1}^\top \mathbf{Z} = \bm{1}^\top$. Comparing \ref{opt:sdp_pw1} and \ref{opt:sdp_awasthi}, identity \eqref{eq:equiv_gram_dist} shows that the two relaxations are equivalent because their objective functions are the same (up to a factor two) while they share the same feasible set.
	\qed
\end{proof}

Next, we establish that $\rsdp$ is at least as tight a relaxation of the cardinality-constrained $K$-means clustering problem~\eqref{opt:kmean_size} as the stronger relaxation \ref{opt:sdp_pw1} of \cite{Peng2007}.
\begin{proposition}
\label{prop:rsdp_vs_pw1}
	We have $\min \mathcal{R}_{\rm{SDP}} \geq \min$ \emph{\ref{opt:sdp_pw1}}.
\end{proposition}

Note that, through Observation~\ref{obs:equi_pw_aw}, Proposition~\ref{prop:rsdp_vs_pw1} also implies that $\rsdp$ is at least as tight as the relaxation \ref{opt:sdp_awasthi} of \cite{Awasthi2015}.

\begin{proof}{Proof of Proposition~\ref{prop:rsdp_vs_pw1}}
	To prove that $\rsdp$ is at least as tight a relaxation as \ref{opt:sdp_pw1}, we will argue that for every feasible solution $\{ (\bm{x}^k, \mathbf{M}^k) \}_{k=1}^K$ of $\rsdp$ one can construct a solution 
\begin{equation*}
\overline{\mathbf{Z}} = \frac{1}{4} \sum_{k=1}^{K} \frac{1}{n_k} \big( \mathbf{M}^k + \bm{1}\bm{1}^\top + \bm{x}^k \bm{1}^\top + \bm{1}(\bm{x}^k)^\top\big)
\end{equation*}
which is feasible in \ref{opt:sdp_pw1} and achieves the same objective value. We first verify the feasibility of the proposed solution $\overline{\mathbf{Z}}$. Note that $\overline{\mathbf{Z}}$ is symmetric by construction. Next, we can directly verify that $\overline{\mathbf{Z}}$ is positive semidefinite since
\begin{equation*}
\begin{aligned}
\overline{\mathbf{Z}} \succeq 0 \quad &\Longleftarrow \quad \mathbf{M}^k + \bm{1}\bm{1}^\top + \bm{x}^k \bm{1}^\top + \bm{1}(\bm{x}^k)^\top \succeq \bm 0 \quad && &&\forall k=1,\ldots,K \\
	&\Longleftrightarrow \quad \hspace{-0.2mm} \bm{v}^\top \big( \mathbf{M}^k + \bm{1}\bm{1}^\top + \bm{x}^k \bm{1}^\top + \bm{1}(\bm{x}^k)^\top \big) \, \bm{v} \geq 0 \quad &&\forall \bm{v} \in \mathbb{R}^N \quad &&\forall k=1,\ldots,K \\
	&\Longleftarrow \quad \hspace{0.5mm}\bm{v}^\top \big( \bm{x}^k(\bm{x}^k)^\top + \bm{1}\bm{1}^\top + \bm{x}^k \bm{1}^\top + \bm{1}(\bm{x}^k)^\top \big) \, \bm{v} \geq 0 \quad &&\forall \bm{v} \in \mathbb{R}^N \quad &&\forall k=1,\ldots,K \\
	&\Longleftrightarrow \quad  \hspace{-0.2mm} \big( \bm{v}^\top \bm{x}^k \big)^2 + \big( \bm{v}^\top \bm{1} \big)^2 + 2 \, \big( \bm{v}^\top \bm{x}^k \big) \big( \bm{v}^\top \bm{1} \big) \geq 0 \quad &&\forall \bm{v} \in \mathbb{R}^N \quad &&\forall k=1,\ldots,K \\
	&\Longleftrightarrow \quad  \hspace{-0.2mm} \big( \bm{v}^\top \bm{x}^k +  \bm{v}^\top \bm{1} \big)^2 \geq 0 \quad &&\forall \bm{v} \in \mathbb{R}^N \quad &&\forall k=1,\ldots,K,
\end{aligned}
\end{equation*}
where the third implication is due to the definition of $\mathcal{C}^{}_{\text{SDP}} (n_k)$, which requires that $\mathbf{M}^k \succeq \bm{x}^k(\bm{x}^k)^\top$. The last statement holds trivially because any quadratic form is non-negative. Next, we can ensure the element-wise non-negativity of $\overline{\mathbf{Z}}$, again through the definition of $\mathcal{C}^{}_{\text{SDP}} (n_k)$:
\begin{equation*}
\begin{aligned}
\overline{\mathbf{Z}} \geq \bm 0 \quad &\Longleftarrow \quad \mathbf{M}^k + \bm{1}\bm{1}^\top + \bm{x}^k \bm{1}^\top + \bm{1}(\bm{x}^k)^\top \geq \bm 0 \quad &&\forall k=1,\ldots,K.
\end{aligned}
\end{equation*}
Furthermore, combining the definition of $\mathcal{C}^{}_{\text{SDP}} (n_k)$ and the constraint $\sum_{k=1}^K \bm{x}^k=(2-K)\bm{1}$ of $\rsdp$, we can see that each row of $\overline{\mathbf{Z}}$ indeed sums up to one:
\begin{equation*}
\begin{aligned}
\overline{\mathbf{Z}} \, \bm{1} \,&=\, \frac{1}{4} \sum_{k=1}^{K} \frac{1}{n_k} (\mathbf{M}^k \bm{1} + \bm{1}\bm{1}^\top \bm{1} + \bm{x}^k \bm{1}^\top \bm{1} + \bm{1}(\bm{x}^k)^\top \bm{1}) \\
	&\,= \frac{1}{4} \sum_{k=1}^{K} \frac{1}{n_k} ((2n_k-N)\bm{x}^k + N \bm{1} + N\bm{x}^k + (2n_k-N)\bm{1}) \\
	&\,= \frac{1}{2} \sum_{k=1}^{K} (\bm{x}^k + \bm{1}) = \bm{1}.
\end{aligned}
\end{equation*}
Finally, the trace of $\overline{\mathbf{Z}}$ is uniquely determined as follows:
\begin{equation*}
\begin{aligned}
\text{Tr}(\overline{\mathbf{Z}}) \, &= \, \frac{1}{4} \sum_{k=1}^{K} \frac{1}{n_k} \text{Tr}\big(\mathbf{M}^k + \bm{1}\bm{1}^\top + \bm{x}^k \bm{1}^\top + \bm{1}(\bm{x}^k)^\top \big) \\
	&=\, \frac{1}{4} \sum_{k=1}^{K} \frac{1}{n_k} \big(2N + 2 (\bm{1}^\top\bm{x}^k) \big) \\
	&=\, \frac{1}{4} \sum_{k=1}^{K} \frac{1}{n_k} \big(2N + 2 (2n_k-N)\big) = K.
\end{aligned}
\end{equation*}
Thus, $\overline{\mathbf{Z}}$ is feasible in \ref{opt:sdp_pw1}, and it remains to prove that it achieves the same objective value as the original solution $\{ (\bm{x}^k, \mathbf{M}^k) \}_{k=1}^K$ in $\rsdp$. Invoking relation \eqref{eq:equiv_gram_dist}, it is easy to see that 
\begin{equation*}
\begin{aligned}
\big\langle \mathbf{W}, \mathbb{I}-\overline{\mathbf{Z}} \big\rangle \,&=\, \frac{1}{2} \big\langle \mathbf{D}, \overline{\mathbf{Z}} \big\rangle \,=\, \frac{1}{8} \bigg\langle \mathbf{D}, \sum_{k=1}^{K} \frac{1}{n_k} (\mathbf{M}^k + \bm{1}\bm{1}^\top + \bm{x}^k \bm{1}^\top + \bm{1}(\bm{x}^k)^\top) \bigg\rangle.
\end{aligned}
\end{equation*}
The proof thus concludes. \qed
\end{proof}

Finally, we assert that $\rsdp^{\rm{b}}$ is at least as tight a relaxation of the balanced $K$-means clustering problem as the corresponding relaxation \ref{opt:sdp_pw3} of \cite{Peng2007}.

\begin{proposition}
\label{prop:rsdp_vs_pw3}
	We have $\min \mathcal{R}_{\rm{SDP}}^{\rm{b}} \geq \min$ \emph{\ref{opt:sdp_pw3}}.
\end{proposition}

\begin{proof}{Proof}
To show that $\rsdp^{\rm{b}}$ is at least as tight a relaxation as \ref{opt:sdp_pw3}, we will again argue that for every feasible solution $\{ (\bm{x}^1, \mathbf{M}^1), (\bm{x}, \mathbf{M})\}$ of $\rsdp^{\rm{b}}$ one can construct a solution
\begin{equation*}
    \overline{\mathbf{Z}} = \frac{K}{4N} \left( (\mathbf{M}^1 + \bm{11}^\top + \bm{x}^1\bm{1}^\top + \bm{1}(\bm{x}^1)^\top) + (K-1) (\mathbf{M} + \bm{11}^\top + \bm{x}\bm{1}^\top + \bm{1}\bm{x}^\top) \right)
\end{equation*}
that is feasible in \ref{opt:sdp_pw3} and achieves the same objective value. Following similar steps as in the proof of Proposition~\ref{prop:rsdp_vs_pw1}, one can verify that $\overline{\mathbf{Z}}$ indeed satisfies $\overline{\mathbf{Z}} \succeq \bm{0}$, $\overline{\mathbf{Z}} \geq \bm{0}$, $\mathbf{\overline{Z}}\bm{1} = \bm{1}$ and $\text{Tr}(\mathbf{\overline{Z}})=K$. In order to see that $\overline{\mathbf{Z}} \leq (K/N)\bm{11}^\top$, note from the definition of $\mathcal{C}_{\rm{SDP}}(n)$ (where $n = N/K$ denotes the shared cardinality of all clusters) that
\begin{equation*}
\begin{aligned}
	&2(\mathbf{M}^1 - \bm{11}^\top) = (\mathbf{M}^1 - \bm{11}^\top + \bm{x}^1 \bm{1}^\top - \bm{1} (\bm{x}^1)^\top) + (\mathbf{M}^1 - \bm{11}^\top - \bm{x}^1 \bm{1}^\top + \bm{1}(\bm{x}^1)^\top) \leq \bm{0} \;\Longrightarrow \; \mathbf{M}^1 \leq \bm{11}^\top, \\
	&2(\mathbf{M} - \bm{11}^\top) = (\mathbf{M} - \bm{11}^\top + \bm{x1}^\top - \bm{1x}^\top) + (\mathbf{M} - \bm{11}^\top - \bm{x1}^\top + \bm{1x}^\top) \leq \bm{0} \;\Longrightarrow \; \mathbf{M} \leq \bm{11}^\top.
\end{aligned}
\end{equation*}
Using this insight and the constraint $\bm{x}^1 + (K-1)\bm{x}=(2-K)\bm{1}$ of $\rsdp^{\rm{b}}$, any arbitrary element $\overline{z}_{ij}$ of $\overline{\mathbf{Z}}$ can be bounded above as desired,
\begin{equation*}
\begin{aligned}
\overline{z}_{ij} &= \frac{K}{4N} \left( (m_{ij}^1 + 1 + x_i^1 + x_j^1) + (K-1)(m_{ij} + 1 + x_i + x_j) \right) \\
	&\leq \frac{K}{4N} \left( 2K + x_i^1 + (K-1) x_i + x_j^1 + (K-1) x_j \right) = \frac{K}{N}.
\end{aligned}
\end{equation*}
Finally, a direct calculation reveals that the objective of $\mathcal{R}^\text{b}_\text{SDP}$ evaluated at $\{ (\bm{x}^1, \mathbf{M}^1), (\bm{x}, \mathbf{M})\}$ coincides with the objective of~\ref{opt:sdp_pw3} evaluated at $\overline{\mathbf{Z}}$, which from \eqref{eq:equiv_gram_dist} is equal to $\frac{1}{2} \langle \mathbf{D}, \overline{\mathbf{Z}} \rangle$. Hence, $\mathcal{R}^\text{b}_\text{SDP}$ is at least as tight a relaxation as~\ref{opt:sdp_pw3}, and the proof concludes. \qed
\end{proof}

}

{\color{black}
Note that while Propositions~\ref{prop:rsdp_vs_pw1} and~\ref{prop:rsdp_vs_pw3} demonstrate that our SDP relaxations $\mathcal{R}_{\rm{SDP}}$ and $ \mathcal{R}_{\rm{SDP}}^{\rm{b}}$ are at least as tight as their respective counterparts by \cite{Peng2007}, similar tightness results cannot be established for our LP relaxations. Indeed, our numerical experiments based on real-world datasets in Section~\ref{sec:numericalexp} show that both $\mathcal{R}_{\rm{LP}}$ and $ \mathcal{R}_{\rm{LP}}^{\rm{b}}$ can be strictly weaker than \ref{opt:sdp_pw1} and \ref{opt:sdp_pw3}, respectively. Furthermore, it is possible to construct artificial datasets on which even \ref{opt:sdp_pw2} outperforms $\mathcal{R}_{\rm{LP}}$ and $ \mathcal{R}_{\rm{LP}}^{\rm{b}}$.
}

\subsection{Perfect Recovery Guarantees}

We now demonstrate that the relaxations of Section~\ref{sec:balanced_clustering} are tight and that Algorithm~\ref{alg:det_rounding_balanced} finds the optimal clustering if the clusters are perfectly separated in the sense of the following assumption.

\paragraph{\textbf{\em{(S)}} Perfect Separation:} \label{assumption:recovery_balanced}
There exists a balanced partition $( J_1, \hdots, J_K )$ of $\{1, \hdots, N\}$ where each cluster $k=1,\hdots,K$ has the same cardinality $\vert J_k \vert = N/K \in \mathbb{N}$, and
\begin{equation*}
	\max_{1 \leq k \leq K} \max_{i,j \in J_k} d_{ij} < \min_{1 \leq k_1 < k_2 \leq K} \min_{i \in J_{k_1}, \, j \in J_{k_2}} d_{ij}.
\end{equation*}

{\color{black}Assumption~{\bf (S)} implies that the dataset admits the natural balanced clustering $( J_1, \hdots, J_K )$, and that the largest cluster diameter (\emph{i.e.}, $\max_{1 \leq k \leq K} \max_{i,j \in J_k} d_{ij}$) is smaller than the smallest distance between any two distinct clusters (\emph{i.e.}, $\min_{1 \leq k_1 < k_2 \leq K} \min_{i \in J_{k_1}, j \in J_{k_2}} d_{ij}$)}.
 
\begin{theorem}
\label{thm:recovery_balanced}
If Assumption {\bf (S)} holds, then the optimal values of~{\em $\rlpb$} and~\ref{opt:kmean_milp} coincide. Moreover, the clustering $( J_1, \hdots, J_K )$ is optimal in~\ref{opt:kmean_milp} and is recovered by Algorithm~\ref{alg:det_rounding_balanced}.
\end{theorem}

{\color{black}Put simply, Theorem \ref{thm:recovery_balanced} states that for datasets whose hidden classes are balanced and well separated, Algorithm 2 will succeed in recovering this hidden, provably optimal clustering.}

\begin{proof}{Proof of Theorem \ref{thm:recovery_balanced}}
	Throughout the proof we assume without loss of generality that the clustering $( J_1, \hdots, J_K )$ from Assumption~{\bf (S)} satisfies $1\in J_1$, that is, the cluster containing the datapoint $\bm \xi_1$ is assigned the number $k=1$. The proof now proceeds in two steps. In the first step, we show that the optimal values of the LP~$\rlpb$ and the MILP~\ref{opt:kmean_milp} are equal and that they both coincide with the sum of squared intra-cluster distances of the clustering $( J_1, \hdots, J_K )$, which amounts to
\begin{equation*}
	\frac{1}{2n} \sum_{k=1}^K \sum_{i,j \in J_k} d_{ij}.
\end{equation*}
In the second step we demonstrate that the output $(I_1,\hdots,I_K)$ of Algorithm~\ref{alg:det_rounding_balanced} coincides with the optimal clustering $( J_1, \hdots, J_K )$ from Assumption~{\bf (S)}. As the algorithm uses the same procedure $K$ times to recover the clusters one by one, it is actually sufficient to show that the first iteration of the algorithm correctly identifies the first cluster, that is, it suffices to prove that $I_1 = J_1$.

\paragraph{Step 1:} For any feasible solution $(\bm{x}^1, \bm{x}, \mathbf{M}^1, \mathbf{M})$ of~$\rlpb$, we define $\mathbf{H},\mathbf{W}\in\mathbb{S}^N$ through
\begin{equation}
\label{eq:weight_matrices} 
	\mathbf{H} = \mathbf{M}^1 + \bm{11}^\top + \bm{x}^1\bm{1}^\top + \bm{1}(\bm{x}^1)^\top \quad \text{and} \quad
	\mathbf{W} = \mathbf{M} + \bm{11}^\top + \bm{x1}^\top + \bm{1x}^\top.
\end{equation} 
From the definition of $\mathcal{C}^{}_{\text{LP}} (n)$ it is clear that $\mathbf{H}, \mathbf{W} \geq \bm 0$. Moreover, we also have that
\begin{equation*}
\begin{aligned}
	\sum_{i \neq j} h_{ij} &= \sum_{i \neq j} m_{ij}^1 + N(N-1) + 2(N-1)(\bm{x}^1)^\top\bm{1} \\
	&= (2n-N)^2 -N + N(N-1) + 2(N-1)(2n-N) = 4n(n-1).
\end{aligned}
\end{equation*}
A similar calculation for $\mathbf{W}$ reveals that $\sum_{i \neq j} w_{ij} = 4n(n-1)$. Next, we consider the objective function of~$\rlpb$, which can be rewritten in terms of $\mathbf{W}$ and $\mathbf{H}$ as
\begin{equation}
\label{eq:kmean_sdp_balanced_objective}
	\frac{1}{8n}\left< \mathbf{D}, \mathbf{H} + (K-1)\mathbf{W} \right> = \frac{1}{8n} \sum_{i \neq j} d_{ij} (h_{ij} + (K-1)w_{ij}).
\end{equation}
The sum on the right-hand side can be viewed as a weighted average of the squared distances $d_{ij}$ with non-negative weights $h_{ij} + (K-1)w_{ij}$, where the total weight is given by
\begin{equation*}
	\sum_{i \neq j} \left( h_{ij} + (K-1)w_{ij} \right) = 4Kn(n-1).
\end{equation*}
Furthermore each weight $h_{ij} + (K-1)w_{ij}$ is bounded above by 4 because
\begin{equation}
\label{eq:kmean_sdp_weight_bound}
\begin{aligned}
	h_{ij} + (K-1)w_{ij} &= (m^1_{ij} + 1 + x^1_i + x^1_j) + (K-1)(m_{ij} + 1 + x_i + x_j) \\
	&\leq 2K + (x^1_i + (K-1)x_i) + (x^1_j + (K-1)x_j ) = 4,
\end{aligned}
\end{equation}
where the inequality holds because $\mathbf{M}^1,\mathbf{M} \leq \bm{11}^\top$ (which we know from the proof of Proposition~\ref{prop:rsdp_vs_pw3}) and the last equality follows from the constraint $\bm{x}^1 + (K-1)\bm{x} = (2-K)\bm1$ in~$\rlpb$. 

Hence, the sum on the right hand side of~\eqref{eq:kmean_sdp_balanced_objective} assigns each squared distance $d_{ij}$ with $i\neq j$ a weight of at most $4$, while the total weight equals $4Kn(n-1)$. A lower bound on the sum is thus obtained by assigning a weight of 4 to the $Kn(n-1)$ smallest values $d_{ij}$ with $i\neq j$. Thus, we have
\begin{equation}
\label{eq:kmean_sdp_balanced_bound}
\begin{aligned}
	\frac{1}{8n}\left< \mathbf{D}, \mathbf{H} + (K-1)\mathbf{W} \right> 
	&\geq 
	\frac{1}{2n} \left\{ \text{sum of the $Kn(n-1)$ smallest entries of $d_{ij}$ with $i \neq j$} \right\} \\
	&=\frac{1}{2n} \sum_{k=1}^K \sum_{i,j \in J_k} d_{ij},
\end{aligned}
\end{equation}
where the last equality follows from Assumption {\bf (S)}. By Lemma~\ref{lem:distance}, the right-hand side of~\eqref{eq:kmean_sdp_balanced_bound} represents the objective value of the clustering $( J_1, \hdots, J_K )$ in the MILP~\ref{opt:kmean_milp}. Thus, $\rlpb$ provides an upper bound on~\ref{opt:kmean_milp}. By Corollary~\ref{cor:kmean_sdp_balanced}, $\rlpb$ also provides a lower bound on~\ref{opt:kmean_milp}. We may thus conclude that the LP relaxation~$\rlpb$ is tight and, {\color{black}as a consequence}, that the clustering $( J_1, \hdots, J_K )$ is indeed optimal in~\ref{opt:kmean_milp}. 

\paragraph{Step 2:} As the inequality in~\eqref{eq:kmean_sdp_balanced_bound} is tight, any optimal solution to~$\rlpb$ satisfies $h_{ij} + (K-1)w_{ij} = 4$ whenever $i\neq j$ and $i,j \in J_k$ for some $k =1, \hdots, K$ ({\em i.e.}, whenever the datapoints $\bm\xi_i$ and $\bm\xi_j$ belong to the same cluster). We will use this insight to show that Algorithm~\ref{alg:det_rounding_balanced} outputs~$I_1 = J_1$. 

For any $i \in J_1$, the above reasoning and our convention that $1\in J_1$ imply that $h_{1i} + (K-1)w_{1i} = 4$. This in turn implies via~\eqref{eq:kmean_sdp_weight_bound} that $m_{1i}^1 = m_{1i} = 1$ for all $i \in J_1$. 

From the definition of $\mathcal{C}^{}_{\text{LP}} (n)$, we know that
\begin{equation*}
\begin{aligned}
	&2(\mathbf{M}^1 + \bm{11}^\top) = (\mathbf{M}^1 + \bm{11}^\top + \bm{x}^1\bm{1}^\top + \bm{1}(\bm{x}^1)^\top) + (\mathbf{M}^1 + \bm{11}^\top - \bm{x}^1\bm{1}^\top - \bm{1}(\bm{x}^1)^\top) \geq \bm 0 \; \Longrightarrow \; \mathbf{M}^1 \geq -\bm{11}^\top.
\end{aligned}
\end{equation*}
This allows us to conclude that
\begin{equation*}
	2n-N = \sum_{i=1}^N m_{1i}^1 = \sum_{i \in J_1} m_{1i}^1 + \sum_{i \notin J_1} m_{1i}^1 \geq n + (N-n)(-1) = 2n-N,
\end{equation*}
where the first equality holds because $\mathbf{M}^1\bm{1} = (2n-N)\bm{x}^1$, which is one of the constraints in $\rlpb$, and because of our convention that $x_1^1 = 1$. Hence, the above inequality must be satisfied as an equality, which in turn implies that $m_{1i}^1 = -1$ for all $i \notin J_1$. 

For any $i\notin J_1$, the $1i$-th entry of the matrix inequality $\mathbf{M}^1 + \bm{1}\bm{1}^\top - \bm{x}^1\bm{1}^\top - \bm{1}(\bm{x}^1)^\top \geq \bm{0}$ from the definition of~$\mathcal{C}^{}_{\text{LP}} (n)$ can be expressed as
\begin{equation*}
\begin{aligned}
	&0 \leq m_{1i}^1 + 1 - x_1^1 - x_i^1 \quad \forall i = 1, \hdots, N \; \Longrightarrow \; x_i^1 \leq -1,
\end{aligned}
\end{equation*}
where the implication holds because $m^1_{1i}=-1$ for $i\notin J_1$ and because $x_1^1=1$ due to the symmetry breaking constraint in $\rlpb$. Similarly, for any $i\in J_1$, the $ii$-th entry of the matrix inequality $\mathbf{M}^1 + \bm{1}\bm{1}^\top - \bm{x}^1\bm{1}^\top - \bm{1}(\bm{x}^1)^\top \geq \bm{0}$ can be rewritten as
\begin{equation*}
\begin{aligned}
	&0 \leq m_{ii}^1 + 1 - 2x_i^1 \quad \forall i = 1, \hdots, N \; \Longrightarrow \; x_i^1 \leq 1,
\end{aligned}
\end{equation*}
where the implication follows from the constraint $\text{diag}(\mathbf{M}^1) = \bm{1}$ in $\rlpb$.

As $x_i^1\leq 1$ for all $i\in J_1$ and $x_i^1\leq -1$ for all $i\notin J_1$, the equality constraint $\bm{1}^\top\bm{x}^1 = 2n-N$ from the definition of $\mathcal{C}^{}_{\text{LP}}(n)$ can only be satisfied if $x_i^1 = 1$ for all $i \in J_1$ and $x_i^1 = -1$ for all $i \notin J_1$. Since Algorithm~\ref{alg:det_rounding_balanced} constructs $I_1$ as the index set of the $n$ largest entries of the vector $\bm x^1$, we conclude that it must output $I_1 = J_1$  and the proof completes. \qed
\end{proof}

Theorem~\ref{thm:recovery_balanced} implies via 
Corollary~\ref{cor:kmean_sdp_balanced} that the optimal values of~\ref{opt:kmean_sdp_balanced} and~\ref{opt:kmean_milp} are also equal. Thus, both the LP and the SDP relaxation lead to perfect recovery.

In the related literature, Assumption~{\bf (S)} has previously been used by~\cite{Elhamifar2012} to show that the natural clustering can be recovered in the context of unconstrained exemplar-based clustering whenever a regularization parameter is chosen appropriately. In contrast, our formulation does not rely on regularization parameters. Likewise, Theorem~\ref{thm:recovery_balanced} is reminiscent of Theorem~9 by~\cite{Awasthi2015} {\color{black} which formalizes the recovery properties of their LP relaxation for the unconstrained $K$-means clustering problem.} \cite{Awasthi2015} assume, however, that the datapoints are drawn independently from a mixture of $K$ isotropic distributions and provide a probabilistic recovery guarantee that improves with~$N$ and deteriorates with~$d$. In contrast, our recovery guarantee for constrained clustering is deterministic, model-free and dimension-independent. {\color{black} If Assumption~{\bf (S)} holds, simpler algorithms than Algorithm~\ref{alg:det_rounding_multi} and~\ref{alg:det_rounding_balanced} can be designed to recover the true clusters. {\color{black} For instance, a simple threshold approach (\emph{i.e.}, assigning datapoints to the same cluster whenever the distance between them falls below a given threshold) would be able to recover the true clusters whenever Assumption {\bf(S)} holds.} It seems unlikely, however, that such approaches would perform well in a setting where Assumption~{\bf (S)} is not satisfied. {\color{black}In fact, \cite{Awasthi2015} show that their LP relaxation fails to recover the true clusters with high probability if Assumption~{\bf (S)} is violated.} In contrast, the numerical experiments of Section~\ref{sec:numericalexp} suggest that Algorithms~\ref{alg:det_rounding_multi} and~\ref{alg:det_rounding_balanced} perform well even if Assumption~{\bf (S)} is violated.

}

\begin{remark}
To our best knowledge, there is no perfect recovery result for the cardinality-constrained $K$-means clustering algorithm by \cite{Bennett2000}, see Appendix, whose performance depends critically on its initialization. To see that it can be trapped in a local optimum, consider the $N = 4$ two-dimensional datapoints $\bm\xi_1 = (0,0)$, $\bm\xi_2 = (a,0)$, $\bm\xi_3 = (a,b)$ and $\bm\xi_4 = (0,b)$ with $0< a < b$, and assume that we seek two balanced clusters. If the algorithm is initialized with the clustering $\{\{1,4\}, \{2,3\}\}$, then this clustering remains unchanged, and the algorithm terminates and reports a suboptimal solution with relative optimality gap $b^2/a^2-1$. In contrast, as Assumption~{\bf (S)} holds, Algorithm~\ref{alg:det_rounding_balanced} recovers the optimal clustering $\{ \{ 1,2 \}, \{ 3,4 \} \}$ by Theorem~\ref{thm:recovery_balanced}. 
\end{remark}

\section{Cardinality-Constrained Clustering with Outliers}
If the dataset is corrupted by outliers, then the optimal value of~\eqref{opt:kmean_size} may be high, indicating that the dataset admits no natural clustering. Note that the bounds from Section~\ref{sec:no-outliers} could still be tight, {\em i.e.}, it is thinkable that the optimal clustering is far from `ideal' even if it can be found with Algorithm~\ref{alg:det_rounding_balanced}. If we gradually remove datapoints that are expensive to assign to any cluster, however, we should eventually discover an `ideal' low-cost clustering. In the extreme case, if we omit all but $K$ datapoints, then the optimal value of~\eqref{opt:kmean_size} drops to zero, and Algorithm~\ref{alg:det_rounding_balanced} detects the optimal clustering due to Theorem~\ref{thm:recovery_balanced}.

We now show that the results of Section~\ref{sec:no-outliers} (particularly Proposition~\ref{prop:kmean_milp} and Theorem~\ref{thm:kmean_sdp}) extend to situations where $n_0$ datapoints must be assigned to an auxiliary {\em outlier cluster} indexed by $k=0$ ($\sum_{k=0}^K n_k = N$), and where neither the distances between outliers and retained datapoints nor the distances between different outliers contribute to the objective function. In fact, we could equivalently postulate that each of the $n_0$ outliers forms a trivial singleton cluster. {\color{black}The use of cardinality constraints in integrated clustering and outlier detection has previously been considered by~\cite{Sanjay_outliers} in the context of local search heuristics. Inspired by this work, we henceforth minimize the sum of squared intra-cluster distances of the $N-n_0$ non-outlier datapoints}. We first prove that the joint outlier detection and cardinality-constrained clustering problem admits an exact MILP~reformulation.

\begin{proposition}[MILP Reformulation]
\label{prop:kmean_milp_outlier}
	The joint outlier detection and cardinality-constrained clustering problem is equivalent to the MILP 
\begin{equation}
\label{opt:kmean_milp_outlier} \tag{$\milpo$}
\begin{aligned}
    &\text{~\em{minimize}} && \textstyle \frac{1}{2} \sum_{k=1}^K \frac{1}{n_k} \sum_{i,j=1}^N d_{ij}\eta_{ij}^k \\
    &\text{~\em{subject to}} && \pi_i^k \in \{ 0, 1 \},~ \eta_{ij}^k \in \mathbb{R}_+ && i,j = 1, \hdots, N,~ k = 0, \hdots, K\\
    &	  && \textstyle\sum_{i=1}^N \pi_i^k = n_k && k = 0, \hdots, K \\
    &	  && \textstyle\sum_{k=0}^K \pi_i^k = 1 && i = 1, \hdots, N \\
    &	  && \eta_{ij}^k \geq \pi_i^k + \pi_j^k - 1 && i,j = 1, \hdots, N,~ k = 0, \hdots, K .
\end{aligned}
\end{equation}
\end{proposition}

\begin{proof}{Proof}
	This is an immediate extension of Proposition~\ref{prop:kmean_milp} to account for the outlier cluster.
	\qed
\end{proof}

In analogy to Section~\ref{sec:relax}, one can demonstrate that the MILP~\ref{opt:kmean_milp_outlier} admits the SDP relaxation
\begin{equation}
\label{opt:kmean_outlier} \tag{$\rsdpo$}
\begin{aligned}
	&\text{~{minimize}} && \textstyle \frac{1}{8} \left< \mathbf{D}, \sum_{k=1}^K \frac{1}{n_k} \left( \mathbf{M}^k + \bm{1} \bm{1}^\top + \bm x^k\bm{1}^\top + \bm{1} (\bm x^k)^\top \right) \right> \\
    &\text{~{subject to}} && (\bm{x}^k, \mathbf{M}^k) \in \mathcal{C}^{}_{\text{{SDP}}}(n_k) \quad k = 0, \hdots, K \\
    &	&&\textstyle\sum_{k=0}^K \bm{x}^k = ( 1 - K ) \bm{1}.
\end{aligned}
\end{equation}
Moreover,~\ref{opt:kmean_outlier} can be further relaxed to an LP, henceforth denoted by $\rlpo$, by replacing the semidefinite representable set $\mathcal{C}^{}_{\text{SDP}} (n_k)$ in~\ref{opt:kmean_outlier} with the polytope $\mathcal{C}^{}_{\text{LP}} (n_k)$ for all $k=0,\hdots,K$.

\begin{theorem}[SDP and LP Relaxations]
\label{thm:kmean_sdp_outlier}
	We have {\em $\min \rlpo$ $\leq \min$ \ref{opt:kmean_outlier} $\leq \min$ \ref{opt:kmean_milp_outlier}}.
\end{theorem}

\begin{proof}{Proof}
	This result generalizes Theorem~\ref{thm:kmean_sdp} to account for the additional outlier cluster. As it requires no fundamentally new ideas, the proof is omitted for brevity.
	\qed
\end{proof}

The relaxations~\ref{opt:kmean_outlier} and $\rlpo$ not only provide a lower bound on~\ref{opt:kmean_milp_outlier}, but they also give rise to a rounding algorithm that recovers a feasible clustering and thus an upper bound on~\ref{opt:kmean_milp_outlier}; see Algorithm~\ref{alg:det_rounding_outlier}. Note that this procedure calls the outlier-unaware Algorithm~\ref{alg:det_rounding_multi} as a subroutine.

\begin{algorithm}
\caption{Rounding algorithm for joint outlier detection and cardinality-constrained clustering}
\label{alg:det_rounding_outlier}
\begin{algorithmic}[1]
\State \textbf{Input:} $\mathcal{I}_0=\{1,\ldots,N\}$ (data indices), $n_k\in\mathbb N, \ k = 0,\ldots,K$ (cluster sizes).
\State Solve~\ref{opt:kmean_outlier} or $\rlpo$ for the datapoints $\bm\xi_i$, $i \in \mathcal{I}_0$, and record the optimal $\bm x^0\in\mathbb R^N$.
\State Determine a bijection $\rho : \mathcal{I}_0\rightarrow \mathcal{I}_0$ such that $x_{\rho(1)}^0 \geq x_{\rho(2)}^0\geq\cdots\geq x_{\rho(N)}^0$.
\State Set $I_0 \leftarrow \{\rho(1), \ldots, \rho(n_0)\}$ and $\mathcal{I}_1 \leftarrow \mathcal{I}_0 \backslash I_0$.
\State Call Algorithm~\ref{alg:det_rounding_multi} with input $(\mathcal{I}_1, \{n_k\}_{k=1}^K)$ 
to obtain $I_1, \hdots, I_K$.
\State \textbf{Output:} $I_0, \hdots, I_K$.
\end{algorithmic}
\end{algorithm}

If all normal clusters are equally sized, {\em i.e.}, $n_k = n$ for $k = 1, \hdots, K$, then~\ref{opt:kmean_outlier} can be replaced~by
\begin{equation}
\label{opt:kmean_sdp_balanced_outlier} \tag{$\rsdpob$}
\begin{aligned}	
	&~\text{minimize}  && \textstyle \frac{K}{8n} \left< \mathbf{D}, \mathbf{M} + \bm{11}^\top + \bm{x1}^\top + \bm{1x}^\top \right> \\
	&~\text{subject to} &&  (\bm{x}, \mathbf{M}) \in \mathcal{C}^{}_{\text{SDP}}(n), \quad (\bm{x}^0, \mathbf{M}^0) \in \mathcal{C}^{}_{\text{SDP}}(n_0), \quad
	    K\bm{x} + \bm{x}^0 = (1-K)\bm{1},
\end{aligned}
\end{equation}
whose size no longer scales with $K$. Similarly,~$\rlpo$ simplifies to the LP~$\rlpob$ obtained from \ref{opt:kmean_sdp_balanced_outlier} by replacing $\mathcal{C}^{}_{\text{SDP}}(n)$ and $\mathcal{C}^{}_{\text{SDP}}(n_0)$ with $\mathcal{C}^{}_{\text{LP}}(n)$ and $\mathcal{C}^{}_{\text{LP}}(n_0)$, respectively. Note that the cardinality $n_0=N - Kn$ may differ from $n$. 

\begin{corollary}[Relaxations for Balanced Clustering]
\label{cor:kmean_sdp_balanced_outlier}
We have~{\em $\min \rlpob$ $\leq \min$~\ref{opt:kmean_sdp_balanced_outlier} $\leq \min$~\ref{opt:kmean_milp_outlier}}.
\end{corollary}

\begin{proof}{Proof}
	This follows from a marginal modification of the argument that led to Corollary~\ref{cor:kmean_sdp_balanced}. 
	\qed
\end{proof}

If the normal clusters are required to be balanced, then Algorithm~\ref{alg:det_rounding_outlier} should be modified as follows. First, in Step~2 the relaxations~\ref{opt:kmean_sdp_balanced_outlier} or~$\rlpob$ can be solved instead of~\ref{opt:kmean_outlier} or~$\rlpo$, respectively. Moreover, in Step~5 Algorithm~\ref{alg:det_rounding_balanced} must be called as a subroutine instead of Algorithm~\ref{alg:det_rounding_multi}.

In the presence of outliers, the perfect recovery result from Theorem~\ref{thm:recovery_balanced} remains valid if the following perfect separation condition is met, which can be viewed as a generalization of Assumption~{\bf (S)}.

\paragraph{\textbf{\em{(S')}} Perfect Separation:} \label{assumption:recovery_outlier}
There exists a partition $( J_0, J_1, \hdots, J_K )$ of $\{1, \hdots, N\}$ where each normal cluster $k=1,\hdots,K$ has the same cardinality $\vert J_k \vert = (N-n_0)/K\in\mathbb N$, while
\begin{equation*}
\max_{1 \leq k \leq K} \max_{i,j \in J_k} d_{ij} < \min_{1 \leq k_1 < k_2 \leq K} \min_{i \in J_{k_1}, j \in J_{k_2}} d_{ij} 
\quad \text{and} \quad  
\max_{1 \leq k \leq K} \max_{i,j \in J_k} d_{ij} < {\color{black} \min_{i \in J_0, \, j \in \{1,\ldots,N\} \setminus \{i\} }} d_{ij}.
\end{equation*}
Assumption~{\bf (S')} implies that the dataset admits the natural outlier cluster $J_0$ and the natural normal clusters $( J_1, \hdots, J_K )$. It also postulates that the diameter of each normal cluster is strictly smaller than (i) the distance between any two distinct normal clusters and (ii) the distance between any outlier and any other datapoint. Under this condition, Algorithm~\ref{alg:det_rounding_outlier} correctly identifies the optimal clustering.

\begin{theorem}
\label{thm:recovery_outlier}
If Assumption {\bf (S')} holds, then the optimal values of~{\em $\rlpob$} and~{\em $\milpo$} coincide. Moreover, the clustering $( J_0, \hdots, J_K )$ is optimal in~{\em $\milpo$} and is recovered by Algorithm~\ref{alg:det_rounding_outlier}.
\end{theorem}

\begin{proof}{Proof}
The proof parallels that of Theorem~\ref{thm:recovery_balanced} and can be divided into two steps. In the first step we show that the LP relaxation~$\rlpob$ for balanced clustering and outlier detection is tight, and in the second step we demonstrate that Algorithm~\ref{alg:det_rounding_outlier} correctly identifies the clusters $(J_0,\ldots,J_K)$. As for the second step, it suffices to prove that the algorithm correctly identifies the outlier cluster $J_0$. Indeed, once the outliers are removed, the residual dataset satisfies Assumption~{\bf(S)}, and Theorem~\ref{thm:recovery_balanced} guarantees that the normal clusters $(J_1,\ldots,J_K)$ are correctly identified with Algorithm~\ref{alg:det_rounding_balanced}.

As a preliminary, note that $(\bm{x},\mathbf{M}) \in \mathcal{C}^{}_{\text{LP}}(n)$ implies
\begin{equation*}
\begin{aligned}
	\text{diag} (\mathbf{M} + \bm{1}\bm{1}^\top + \bm{x}\bm{1}^\top + \bm{1}\bm{x}^\top) \geq\bm 0  &&\hspace{-2mm}\Longrightarrow \; \bm x \geq -\bm 1,\\
	\text{diag} (\mathbf{M} + \bm{1}\bm{1}^\top - \bm{x}\bm{1}^\top - \bm{1}\bm{x}^\top) \geq\bm 0  &&\hspace{-2mm}\Longrightarrow \; \bm x \leq +\bm 1,
\end{aligned}
\end{equation*}
where the implications use $\text{diag} (\mathbf{M})=\bm 1$. Similarly, $(\bm{x}^0,\mathbf{M}^0) \in \mathcal{C}^{}_{\text{LP}}(n_0)$ implies $-\bm{1} \leq \bm{x}^0 \leq +\bm{1}$.

\paragraph{Step 1:}
For any feasible solution $(\bm{x}^0, \bm{x}, \mathbf{M}^0, \mathbf{M})$ of~$\rlpob$, introduce the auxiliary matrix $\mathbf{H} = \mathbf{M} + \bm{11}^\top + \bm{1x}^\top + \bm{x1}^\top$. Recall from the proof of Theorem~\ref{thm:recovery_balanced} that $\mathbf{H} \geq \bm{0} $ and 
\begin{equation*}
\begin{aligned}
	\sum_{i \neq j} h_{ij} = 4n(n-1). 
\end{aligned}
\end{equation*}
The constraint $K\bm{x} + \bm{x}^0 = (1-K)\bm{1}$ from $\rlpob$ ensures via the inequality $-\bm 1\leq \bm x^0$ that $\bm x\leq (\tfrac{2}{K}-1)\bm 1$. Recalling from the proof of Theorem~\ref{thm:recovery_balanced} that $\mathbf{M} \leq \bm{11}^\top$, we then find
\begin{equation}
\label{eq:bound_on_H_and_x}
	h_{ij} = m_{ij} + 1 + x_i + x_j \leq 1 + 1 + \left( \frac{2}{K}-1 \right) + \left( \frac{2}{K}-1 \right) = \frac{4}{K}\quad \forall i,j = 1,\hdots, N.
\end{equation}
Similar arguments as in the proof of Theorem~\ref{thm:recovery_balanced} reveal that the objective function of the joint outlier detection and (balanced) clustering problem~$\rlpob$ can be expressed as
\begin{equation}
\label{eq:kmean_sdp_balanced_outlier_bound}
\begin{aligned}
	\frac{K}{8n}\left< \mathbf{D}, \mathbf{H} \right> 
	&\geq 
	\frac{1}{2n} \left\{ \text{sum of the $Kn(n-1)$ smallest entries of $d_{ij}$ with $i \neq j$} \right\} \\
	&=\frac{1}{2n} \sum_{k=1}^K \sum_{i,j \in J_k} d_{ij},
\end{aligned}
\end{equation}
where the equality follows from Assumption {\bf (S')}. By Lemma~\ref{lem:distance}, the right-hand side of~\eqref{eq:kmean_sdp_balanced_outlier_bound} represents the objective value of the clustering $( J_0, \hdots, J_K )$ in the MILP~\ref{opt:kmean_milp_outlier}. Thus, $\rlpob$ provides~an upper bound on~\ref{opt:kmean_milp_outlier}. By Corollary~\ref{cor:kmean_sdp_balanced_outlier}, $\rlpob$ also provides a lower bound on~\ref{opt:kmean_milp_outlier}. We may thus conclude that the LP relaxation~$\rlpob$ is tight and, {\color{black}as a consequence}, that the clustering $( J_0, \hdots, J_K )$ is indeed optimal in \ref{opt:kmean_milp_outlier}. 

\paragraph{Step 2:}
As the inequality in~\eqref{eq:kmean_sdp_balanced_outlier_bound} is tight, any optimal solution to~$\rlpob$ satisfies $h_{ij} = \frac{4}{K}$ whenever $i\neq j$ and $i,j \in J_k$ for some $k =1, \hdots, K$ ({\em i.e.}, whenever $\bm\xi_i$ and $\bm\xi_j$ are {\em not} outliers and belong to the same cluster). This in turn implies via~\eqref{eq:bound_on_H_and_x} that $x_i = \frac{2}{K} - 1$ for all $i\in\cup_{k=1}^K J_k$. Furthermore, the constraint $\bm{1}^\top\bm{x} = 2n-N$ from $\mathcal{C}^{}_{\text{LP}}(n)$ implies
\begin{equation*}
	2n - N = \sum_{k=1}^K \sum_{i \in J_k} x_i + \sum_{i \in J_0} x_i 
	\geq Kn \left( \frac{2}{K} - 1 \right) + \sum_{i \in J_0} (-1) 
	= 2n - N,
\end{equation*}
where the inequality holds because $-\bm 1\leq \bm x$. Thus, the above inequality must in fact hold as an equality, which implies that $x_i=-1$ for all $i\in J_0$. The constraint $K\bm{x} + \bm{x}^0 = (1-K)\bm{1}$ from $\rlpob$ further implies that $x^0_i=-1$ for all $i\in\cup_{k=1}^K J_k$ and $x^0_i=+1$ for all $i\in J_0$. 

Since Algorithm~\ref{alg:det_rounding_outlier} constructs $I_0$ as the index set of the $n_0=N-Kn$ largest entries of the vector~$\bm x^0$, we conclude that it must output $I_0 = J_0$, and the proof completes. \qed
\end{proof}
{\color{black}
\begin{remark}[Unknown Cluster Cardinalities]
The joint outlier detection and cardinality-constrained clustering problem~{\em $\milpo$} can also be used when the number of outliers is not precisely known and only an estimate of the relative size (as opposed to the exact cardinality) of the clusters is available. To this end, we solve~{\em $\milpo$} for different values of $n_0$, respectively assigning the remaining $N-n_0$ datapoints to clusters whose relative sizes respect the available estimates. The value $n_0^\star$ representing the most reasonable number of outliers to be removed from the dataset can then be determined using the elbow method{\color{black}; see, \emph{e.g.}, \citet[Chapter~10]{Gareth2017}}.
\end{remark}
}

{\color{black}
As an illustration, consider again the dataset depicted in Figure~\ref{fig:outliers} which showcases the crux of outlier detection in the context of cardinality-constrained clustering.
In Section 1, we inadvertently assumed to have the knowledge that the dataset under consideration was contaminated by three outliers. To demonstrate the practical usefulness of our approach, we will now employ the elbow method to determine the number of outliers $n_0$ without making any assumptions about the dataset. As elucidated in Remark 4, the ideal value of $n_0$ can be determined by solving problem~{\em $\milpo$} repeatedly. However, as~{\em $\milpo$} constitutes an intractable optimization problem, we solve its convex relaxations~$\rlpob$ and~$\rsdpob$ instead and plot the resulting objective values in logarithmic scale in Figure~\ref{fig:elbow}. It becomes apparent that $n_0^\star = 3$ is most appropriate as it marks the transition from the initially steep decline pattern of the objective value to a substantially flatter decline pattern. Note that $n_0$ needs to be a multiple of $K = 3$ to allow for balanced clustering.

\begin{figure}[h]
 	\centering
	\includegraphics[width=0.50\paperwidth]{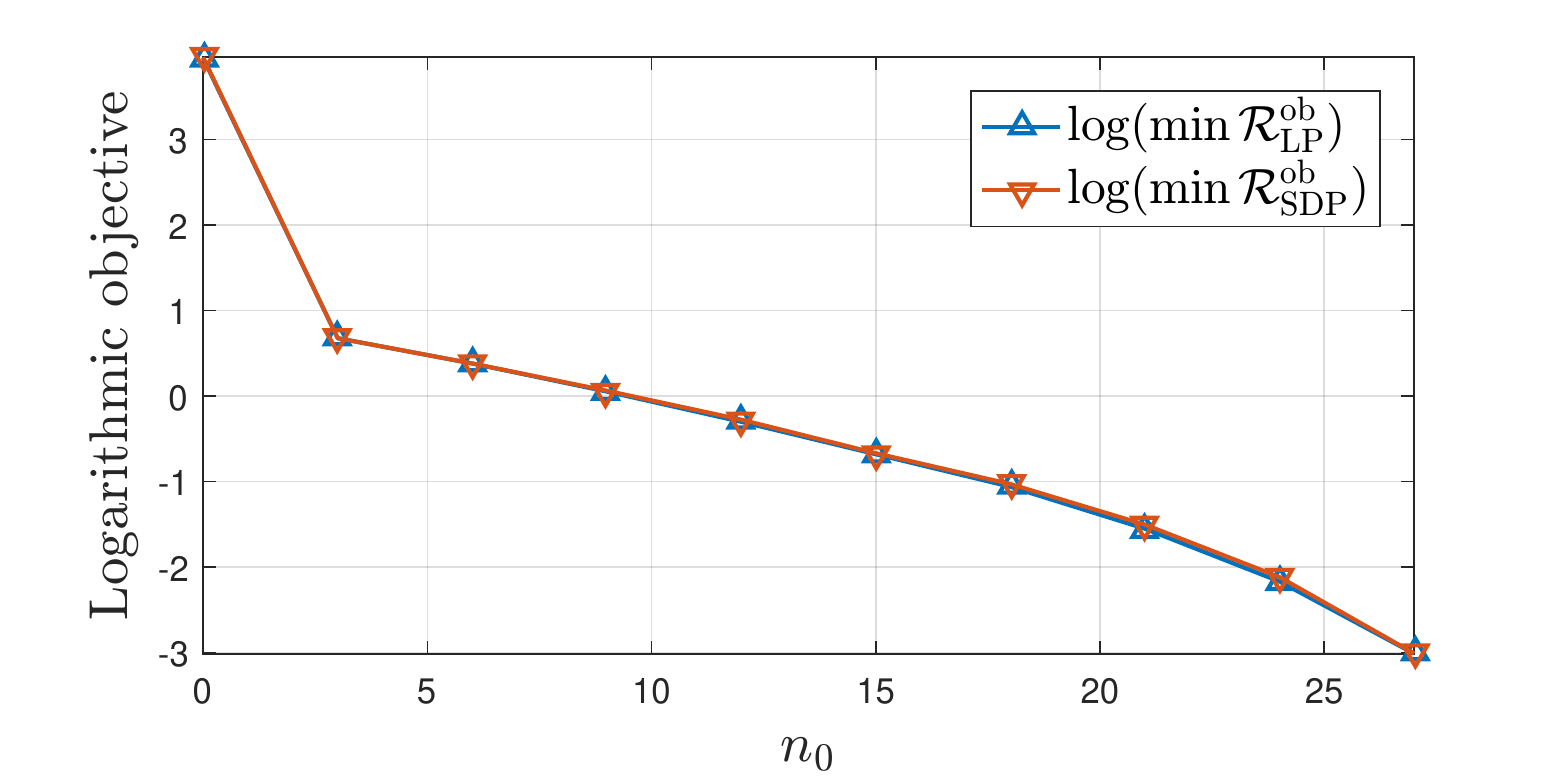} 
	\caption{Elbow plot for the dataset depicted in Figure~\ref{fig:outliers}.}
	\label{fig:elbow}
\end{figure}
}

\section{Numerical Experiments}
\label{sec:numericalexp}

We now investigate the performance of our algorithms on synthetic as well as real-world clustering problems with and without outliers. All LPs and SDPs are solved with CPLEX 12.7.1 and MOSEK 8.0, respectively, using the YALMIP interface on a 3.40GHz i7 computer with 16GB~RAM.

\subsection{Cardinality-Constrained $K$-Means Clustering (Real-World Data)}
We compare the performance of our algorithms from Section~\ref{sec:no-outliers} with the algorithm of \cite{Bennett2000}, see Appendix, and with the two SDP relaxations proposed by \cite{Peng2007} on the classification datasets of the UCI Machine Learning Repository (\url{http://archive.ics.uci.edu/ml/}) with 150--300 datapoints, up to 200 continuous attributes and no missing values. {\color{black} Table~\ref{tab:characteristics} reports the main characteristics of these datasets.} In our experiments, we set the cluster cardinalities to the numbers of true class occurrences in each dataset. {\color{black} It should be emphasized that, in contrast to the other methods, and with exception of the two balanced datasets, the SDP relaxations of \cite{Peng2007} do not have access to the cluster cardinalities. They should thus be seen as a baseline for the performance of the other methods.} {\color{black} Furthermore, we remark that all datasets severely violate Assumption~{\bf (S)}. Indeed, the ratios of largest cluster diameter to smallest distance between clusters (when the clusters are determined by the true labels) vary from 7 to 149, while they should be smaller than one in order to satisfy Assumption~{\bf (S)}. Also, only two datasets actually entail balanced clusters.}

Table~\ref{tab:performance} reports the lower bounds provided by $\rlp / \rlpb$ and $\rsdp / \mathcal{R}_{\text{SDP}}^{\text{b}}$ (LB), the upper bounds from Algorithms~1 and~2 (UB), the objective value of the best of 10 runs of the algorithm of Bennett et al.~(UB), randomly initialized by the cluster centers produced by the \emph{K-means++ algorithm} of \cite{Arthur2007}, the coefficient of variation across these 10 runs (CV), the respective lower bounds (LB) obtained from the SDP relaxations {\color{black} \ref{opt:sdp_pw1}/\ref{opt:sdp_pw3} and~\ref{opt:sdp_pw2} of \cite{Peng2007}{\color{black}, and the solution times for each of these methods. The latter was limited to a maximum of three hours, and in one case (namely, ``Glass Identification''), $\rsdp$ did not terminate within this limit. The ``--" signs in Table~\ref{tab:performance} indicate this occurrence.} 

The obtained lower bounds of $\rsdp$/$\mathcal{R}_{\rm{SDP}}^{\rm{b}}$ allow us to certify that the algorithm of \cite{Bennett2000} provides nearly optimal solutions in almost all instances. Also, both Algorithms~1 and~2 are competitive {\color{black}in terms of solution quality} with the algorithm of \cite{Bennett2000} while providing rigorous error bounds. Moreover, as expected in view of Propositions~\ref{prop:rsdp_vs_pw1}~and~\ref{prop:rsdp_vs_pw3}, for all datasets $\rsdp$/$\mathcal{R}_{\rm{SDP}}^{\rm{b}}$ yield better lower bounds than the SDP relaxations \ref{opt:sdp_pw1}/\ref{opt:sdp_pw3} and \ref{opt:sdp_pw2} of \cite{Peng2007}. The lower bounds obtained from $\rlp$/$\mathcal{R}_{\rm{LP}}^{\rm{b}}$ are competitive with those provided by the relaxations \ref{opt:sdp_pw1}/\ref{opt:sdp_pw3}, and they are always better than the lower bounds provided by their relaxation \ref{opt:sdp_pw2}. It should be mentioned, however, that one can construct instances where the situation is reversed, \emph{i.e.}, both \ref{opt:sdp_pw1}/\ref{opt:sdp_pw3} and \ref{opt:sdp_pw2} are tighter than $\rlp$/$\mathcal{R}_{\rm{LP}}^{\rm{b}}$.} \cite{Peng2007} also suggest a procedure to compute a feasible clustering (and thus upper bounds) for the unconstrained $K$-means clustering problem. However, this procedure relies on an enumeration of all possible Voronoi partitions, which is impractical for $K\geq3$; see \cite{Inaba1994}. Furthermore, it is not clear how to impose cardinality constraints in this setting.

\begin{table}[H]
\scriptsize
\centering
{\color{black}
\begin{tabular}{c|l|c|c|c|l|c}
	& & $N$ & $d$ & $K$ & \multicolumn{1}{c|}{$n_k$} & \\
    ID & Dataset Name & (\# datapoints) & (\# dimensions) & (\# clusters) & \multicolumn{1}{c|}{(cardinalities)} & balanced \\ \hline 
	1 & Iris & 150 & 4 & 3 & 50, 50, 50 & yes \\
	2 & Seeds & 210 & 7 & 3 & 70, 70, 70 & yes \\
	3 & Planning Relax & 182 & 12 & 2 & 130, 52 & no \\ 
	4 & Connectionist Bench & 208 & 60 & 2 & 111, 97 & no \\
	5 & Urban Land Cover & 168 & 147 & 9 & 23, 29, 14, 15, 17, 25, 16, 14, 15 & no \\
	6 & Parkinsons & 195 & 22 & 2 & 48, 147 & no \\ 
	7 & Glass Identification & 214 & 9 & 6 & 70, 76, 17, 13, 9, 29 & no \\  \hline
\end{tabular}
~\\~\\
\caption{Overview of the main dataset characteristics.}
\label{tab:characteristics}
}
\end{table}

\begin{table}[H]
\scriptsize
\centering
{\color{black}
\begin{tabular}{c|ccc|ccc|ccc|cc|cc}
    &\multicolumn{3}{c|}{$\rlp / \rlpb$} & \multicolumn{3}{c|}{$\rsdp / \mathcal{R}_{\text{SDP}}^{\text{b}}$} & \multicolumn{3}{c|}{Bennett et al.} & \multicolumn{2}{c|}{$\mathcal{PW}_{1} /\mathcal{PW}_{1}^{\rm{b}}$} & \multicolumn{2}{c}{$\mathcal{PW}_{2}$} \\
    ID & UB & LB & time [s] & UB & LB & time [s] & UB & CV [\%] & time [s] & LB & time [s] & LB & time [s] \\ \hline 
	1 & 81.4 & 78.8 & 17 & 81.4 & 81.4 & 584 & 81.4 & 0.0 & 6 & 81.4 & 154 & 15.2 & 0.02 \\
	2 & 620.7 & 539.0 & 46 & 605.6 & 605.6 & 3,823 & 605.6 & 0.0 & 7 & 604.5 & 1,320 & 19.0 & 0.03 \\
	3 & 325.9 & 297.0 & 24 & 315.7 & 315.7 & 2,637 & 315.8 & 0.3 & 9 & 299.0 & 510 & 273.7 & 0.02  \\ 
	4 & 312.6 & 259.1 & 49 & 280.6 & 280.1 & 3,638 & 280.6 & 0.4 & 6 & 270.0 & 1,376 & 246.2 & 0.04 \\
	5 & 3.61$\mathrm{e}$9 & 3.17$\mathrm{e}$9 & 2,241 & 3.54$\mathrm{e}$9 & 3.44$\mathrm{e}$9 & 10,754 & 3.64$\mathrm{e}$9 & 9.2 & 13 & 2.05$\mathrm{e}$9 & 460 & 1.94$\mathrm{e}$8 & 0.02 \\
	6 & 1.36$\mathrm{e}$6 & 1.36$\mathrm{e}$6 & 22 & 1.36$\mathrm{e}$6 & 1.36$\mathrm{e}$6 & 2,000 & 1.36$\mathrm{e}$6 & 15.1 & 7 & 1.11$\mathrm{e}$6 & 777 & 6.31$\mathrm{e}$5 & 0.02 \\ 
	7 & 469.0 & 377.2 & 232 & -- & -- & -- & 438.2 & 28.4 & 13 & 321.9 & 1,500 & 23.8 & 0.03 \\  \hline
\end{tabular}
~\\~\\
\caption{Performance of $\rlp$, $\rsdp$, Bennett et al., and Peng and Wei. {\color{black} The ``--'' signs indicate that \\ the problem instance could not be solved within a time limit of three hours.}}
\label{tab:performance}
}
\end{table}

{\color{black} Specifically, in the context of the two balanced datasets (\emph{i.e.}, ``Iris'' and ``Seeds''), we can enrich the preceding {\color{black}comparison} with the heuristics proposed by \cite{Costa2017} and \cite{Malinen2014}. As for the variable neighbourhood search method of \cite{Costa2017}, we were provided with the executables of the \emph{C++} implementation used in that paper. For the ``Iris'' dataset, the best objective value out of 10 independent runs of this method was 81.4 (which is provably optimal thanks to the lower bounds provided by $\rsdp$ and $\mathcal{PW}_{1}^{\rm{b}}$) and the time to execute all runs was 0.12 seconds. For the ``Seeds'' dataset, the best objective value out of 10 independent runs was 605.6 (again, provably optimal in view of the lower bound provided by $\rsdp$) and the overall runtime was 0.53 seconds. The algorithm of \cite{Malinen2014} follows the same steps as the one of \cite{Bennett2000} with the improvement that the cluster assignment step is solved by the Hungarian algorithm, which provides better runtime guarantees and typically solves faster than interior-point methods for LPs. For this reason, the upper bounds of \cite{Malinen2014} for the ``Iris'' and ``Seeds'' dataset coincide with those of \cite{Bennett2000} while their algorithm can be expected to terminate faster. A direct comparison of the time complexity of these two methods can be found in \cite{Malinen2014}.
}

\subsection{Cardinality-Constrained $K$-Means Clustering (Synthetic Data)}
{\color{black}We now randomly generate partitions of $10$, $20$ and $70$ datapoints in $\mathbb{R}^2$ that are drawn from uniform distributions over $K = 3$ unit balls centered at $\bm\zeta_1, \bm\zeta_2$ and $\bm\zeta_3$, respectively, such that $\Vert \bm\zeta_1 - \bm\zeta_2 \Vert = \Vert \bm\zeta_1 - \bm\zeta_3 \Vert = \Vert \bm\zeta_2 - \bm\zeta_3 \Vert = \delta$}. Theorem~\ref{thm:recovery_balanced} shows that $\rlpb$ is tight and that Algorithm~\ref{alg:det_rounding_balanced} can recover the true clusters whenever $n_1 = n_2 = n_3$ and $\delta \geq 4$. Figure~\ref{fig:galaxy} demonstrates that in practice, perfect recovery is often achieved~by Algorithm~\ref{alg:det_rounding_multi} even if $\delta \ll 4$ and $n_1 \neq n_2 \neq n_3$. We also note that $\rsdp$ outperforms $\rlp$ when $\delta$ is small, and that the algorithm of Bennett et al.~frequently fails to determine the optimal solution even if it is run 10 times. {\color{black} In line with the results from the real-world datasets, $\rsdp$ and $\rlp$ are tighter than the stronger SDP relaxation of \cite{Peng2007}. Furthermore, it can be shown that in this setting the weaker relaxation of \cite{Peng2007} always yields the trivial lower bound of zero.} The average runtimes are 7s ($\rlp$), 106s ($\rsdp$), 11s (Bennett et~al.) {\color{black} and 15.6s (Peng and Wei).}

\begin{figure}[h]
 	\centering
	\includegraphics[width=0.36\paperwidth]{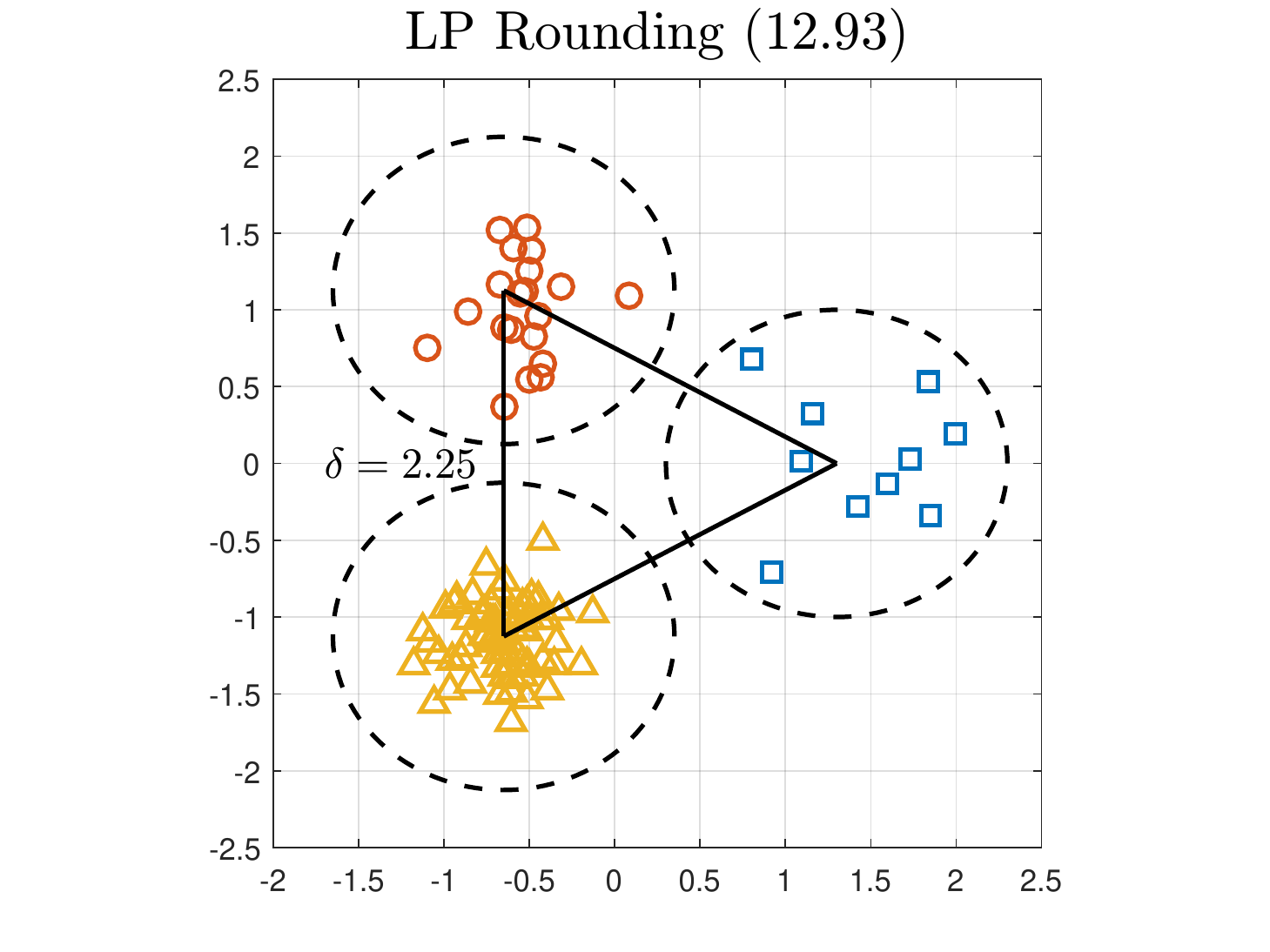} 
	\includegraphics[width=0.36\paperwidth]{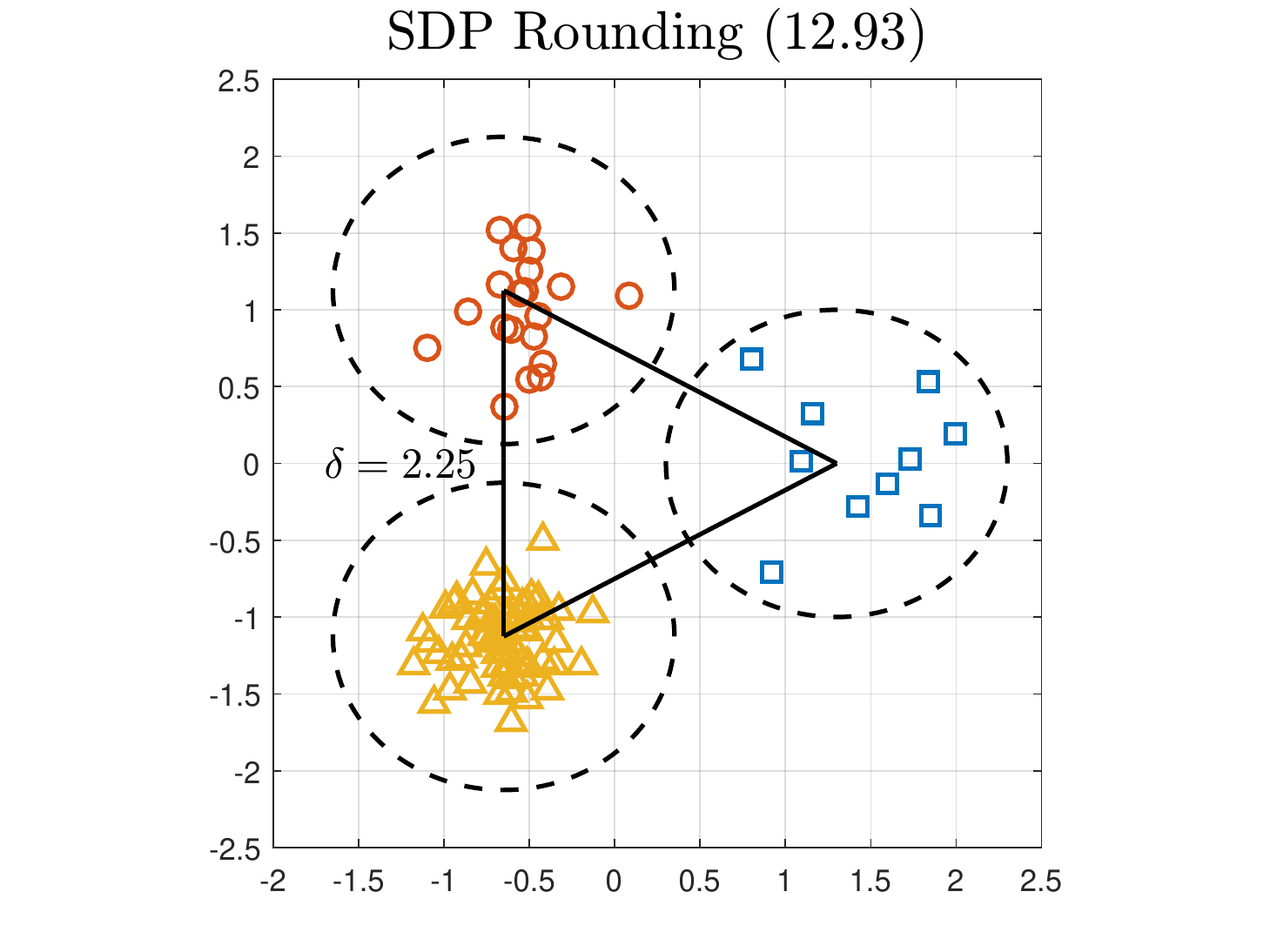} 
	\includegraphics[width=0.36\paperwidth]{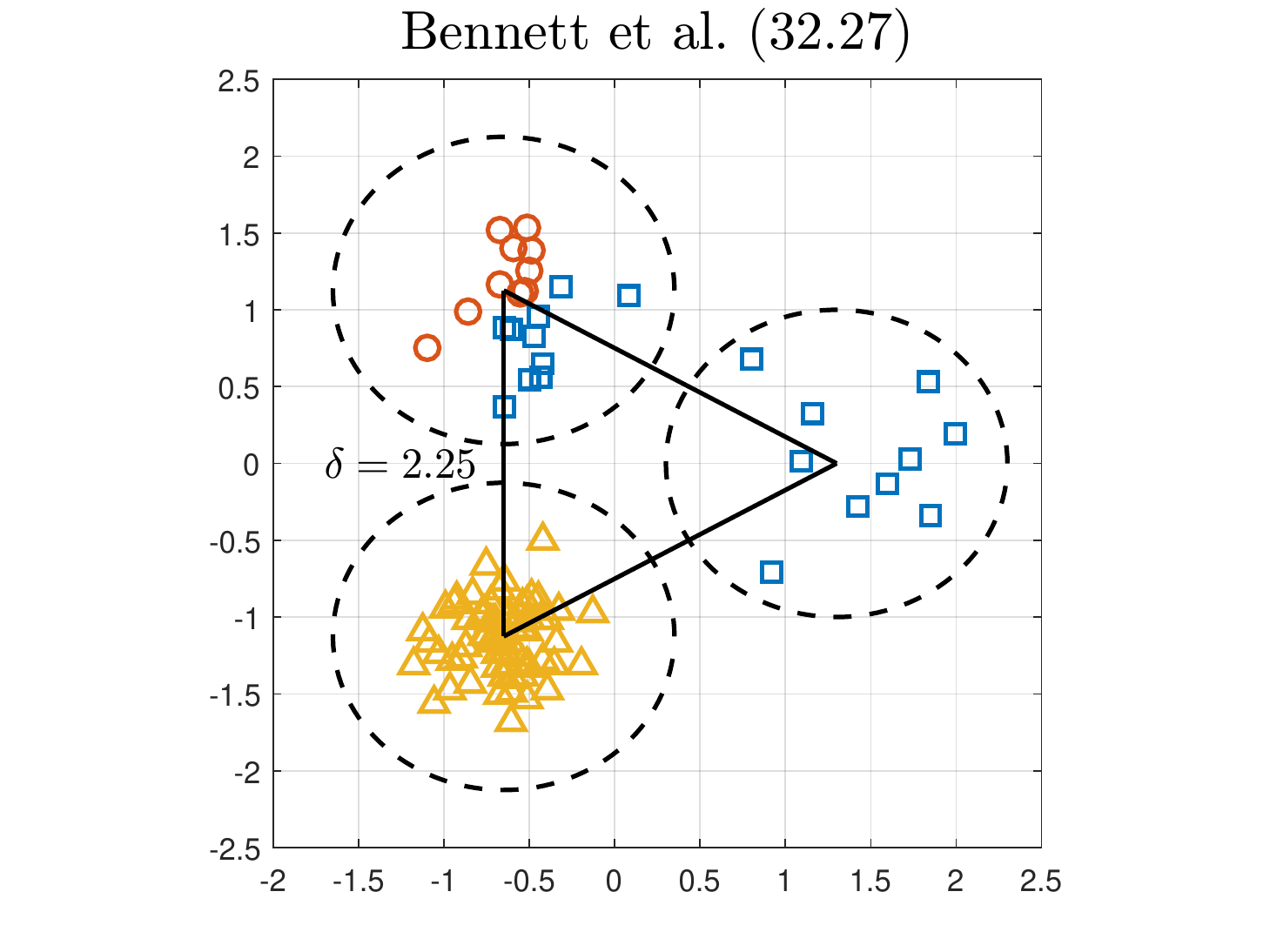} 
	\includegraphics[width=0.36\paperwidth]{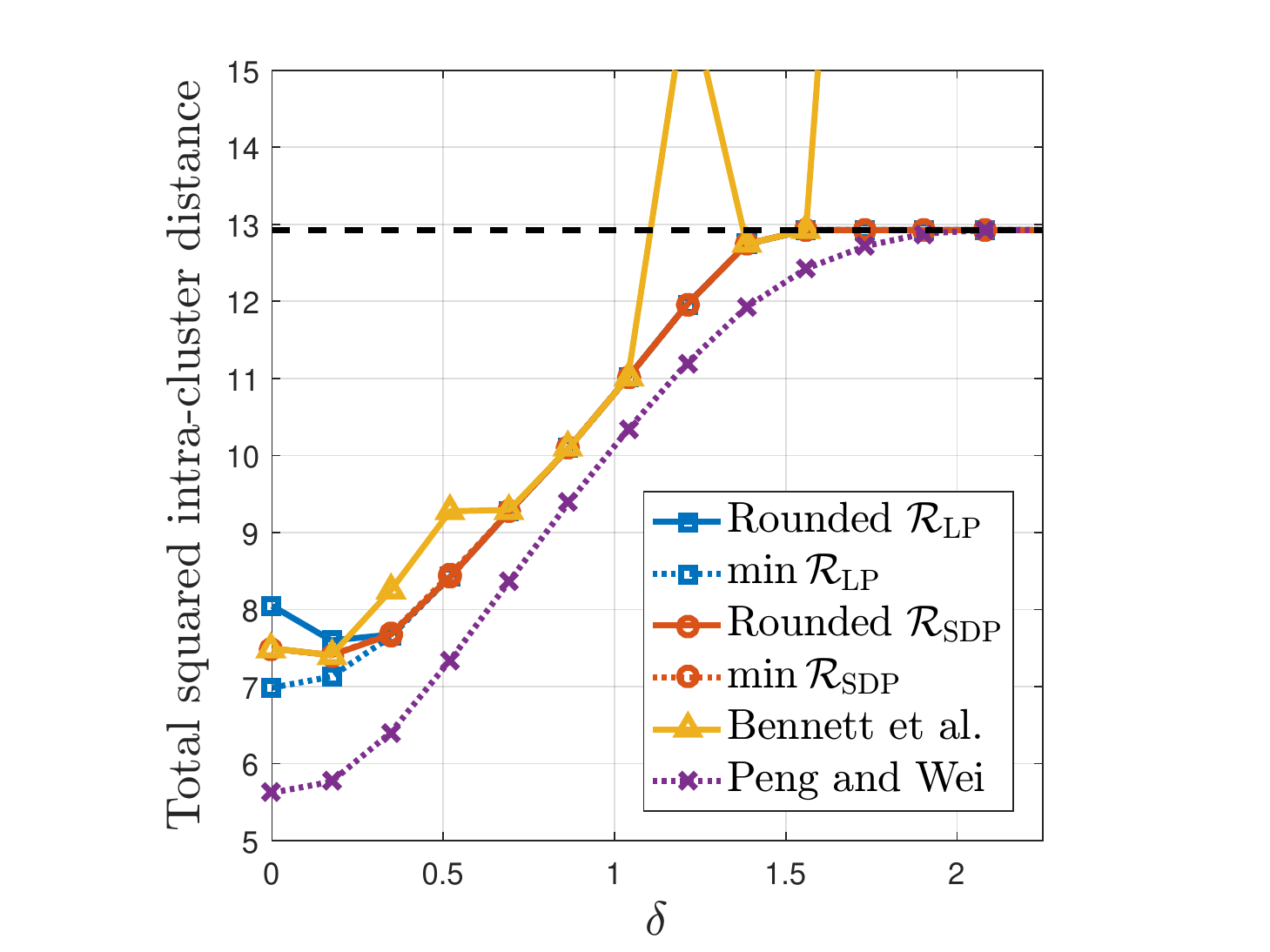} 
	\caption{Comparison between different algorithms for (cardinality-constrained) $K$-means clustering for 100 datapoints where the cardinalities are given by $(n_1, n_2, n_3) = (10, 20, 70)$. Indicated in parentheses next to the panel titles are the respectively achieved sums of squared intra-cluster distances.}
	\label{fig:galaxy}
\end{figure}

\subsection{Outlier Detection}
We use $\rlpo$ and Algorithm~\ref{alg:det_rounding_outlier} to classify the \emph{Breast Cancer Wisconsin (Diagnostic)} dataset. The dataset has $d = 30$ numerical features, which we standardize using a Z-score transformation, and it contains 357 benign and 212 malignant cases of breast cancer. We interpret the malignant cases as outliers and thus set $K = 1$. Figure~\ref{fig:breast_cancer} reports the prediction accuracy as well as the false positives (benign cancers classified as malignant) and false negatives (malignant cancers classified as benign) as we increase the number of outliers $n_0$ from $0$ to $400$. The figure shows that while setting $n_0 \approx 212$, the true number of malignant cancers, maximizes the prediction accuracy, any choice $n_0 \in [156,280]$ leads to a competitive prediction accuracy above 80\%. Thus, even rough estimates of the number of malignant cancer datapoints can lead to cancer predictors of decent quality. The average runtime is 286s, and the optimality gap is consistently below 3.23\% for all values of $n_0$.

\begin{figure}[h]
 \centering
 \includegraphics[width=0.38\paperwidth]{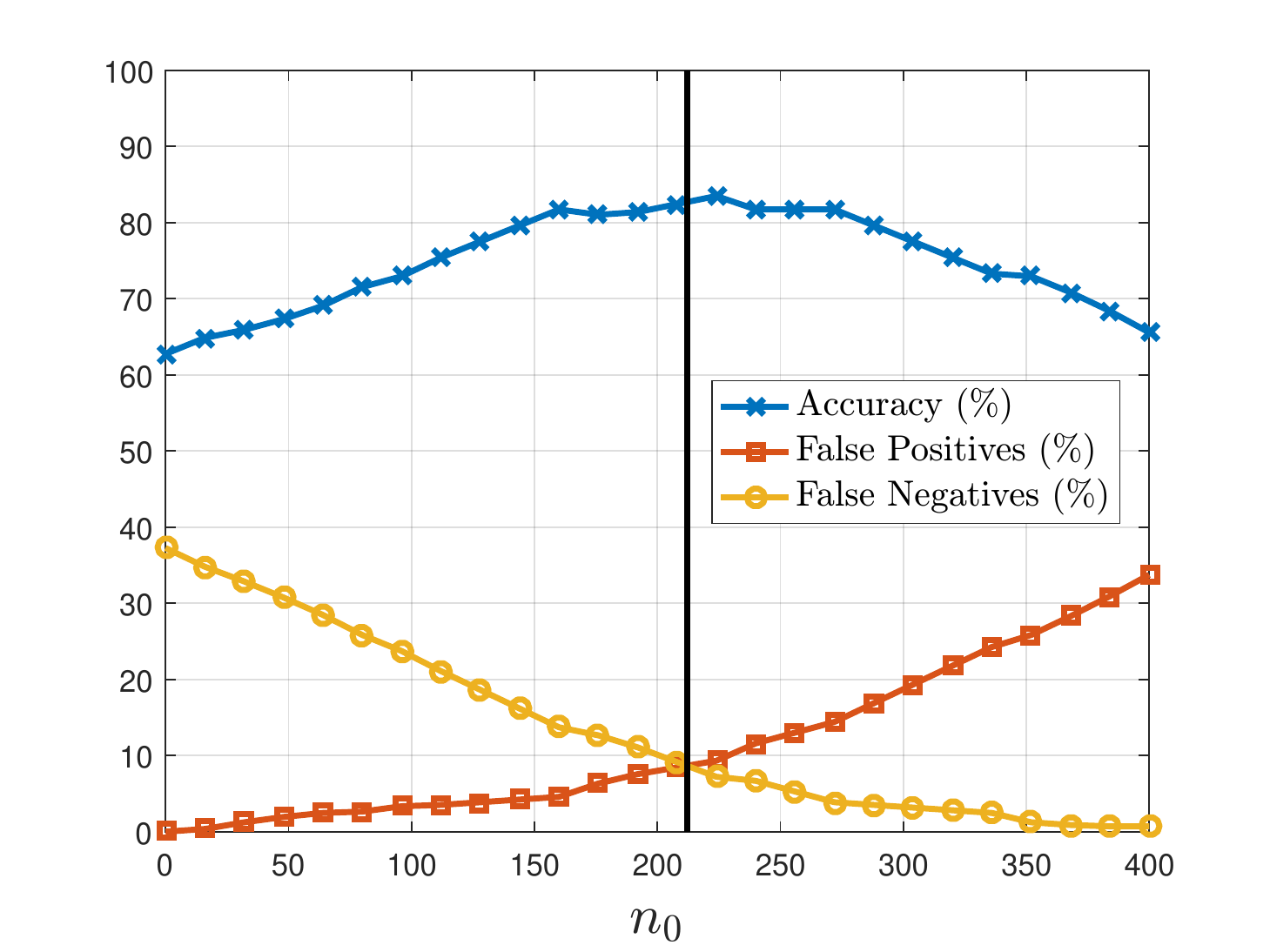}
 \caption{Outlier detection for breast cancer diagnosis.}
 \label{fig:breast_cancer}
\end{figure}

{\color{black}
\section{Conclusion}
\label{sec:conclusion}

Clustering is a hard combinatorial optimization problem. For decades, it has almost exclusively been addressed by heuristic approaches. Many of these heuristics have proven to be very successful in practice as they often provide solutions of high, or at least satisfactory, quality within attractive runtimes. The common drawback of these methods is that there is typically no way of certifying the optimality of the provided solutions nor to give guaranteed bounds on their suboptimality.

Maybe precisely because of this shortcoming, more recently, convex optimization approaches have been proposed for solving relaxed versions of the clustering problem. These conic programs are polynomial-time solvable and offer bounds on the suboptimality of a given solution. Furthermore, the solutions of these conic relaxations can be ``rounded'' to obtain actually feasible solutions to the original clustering problem, which results in a new class of heuristic methods.

The results presented in this paper follow precisely this recent paradigm. Combined, conic relaxations and (rounding) heuristics offer solutions to the clustering problem together with a-posteriori guarantees on their optimality. Naturally, one would also wish for attractive a-priori guarantees on the performance of these combined methods. The conditions required to derive such a-priori guarantees are still quite restrictive, but the strong performance of these methods on practical instances makes us confident that this is a promising avenue for future research.

}

\paragraph{Acknowledgements:} This research was supported by the Singapore Ministry of Education Academic Research Fund Tier 1 R-266-000-131-133, the SNSF grant BSCGI0\texttt{\char`_}157733 and the EPSRC grant EP/N020030/1. {\color{black}We are grateful to Leandro Costa and Daniel Aloise for sharing with us the executables of the variable neighbourhood search heuristic proposed in \cite{Costa2017}.}

\bibliography{references}
\bibliographystyle{ormsv080}

\section*{Appendix: Algorithm of \cite{Bennett2000}}

{\color{black}The algorithm of \cite{Bennett2000} is designed for a variant of problem~\eqref{opt:kmean_size}, where only lower bounds on the clusters' cardinalities are imposed. This algorithm has a natural extension to our cardinality-constrained clustering problem~\eqref{opt:kmean_size} as follows.}
\begin{algorithm}[H]
\caption{Algorithm of Bennett et al.~for cardinality-constrained clustering}
\label{alg:lloyd_lap}
\begin{algorithmic}[1]
\State \textbf{Input:} $\mathcal{I}_1 = \{1,\hdots,N\}$ (data indices), $n_k \in \mathbb{N}, k=1,\hdots,K$ (cluster sizes).
\State Generate the cluster centers $\bm\zeta_1, \hdots, \bm\zeta_K \in \mathbb{R}^d$.
\State Solve the linear assignment problem
\begin{equation*}
	\bm\Pi^\star \in
	\underset{\mathbf{\Pi}}{\text{argmin}} \left\{ 
		\sum_{i = 1}^N \sum_{k = 1}^K \pi_i^k \Vert \bm\xi_i - \bm\zeta_k \Vert^2 :~
		\pi_i^k \in \{ 0, 1\},~
		\sum_{i = 1}^N \pi_i^k = n_k \; \forall k,~
		\sum_{k = 1}^K \pi_i^k = 1 \; \forall i
	\right\}.
\end{equation*}
\State Set $I_k \leftarrow \{ i: (\pi^\star)_i^k = 1 \}$ for all $k = 1, \hdots, K$.
\State Set $\bm\zeta_k \leftarrow \frac{1}{n_k}\sum_{i \in I_k} \bm\xi_i$ for all $k = 1, \hdots, K$.
\State Repeat Steps 3--5 until there are no more changes in $\bm\zeta_1, \hdots, \bm\zeta_K$.
\State \textbf{Output:} $I_1, \hdots, I_K.$
\end{algorithmic}
\end{algorithm}
Algorithm~\ref{alg:lloyd_lap} adapts a classical local search heuristic for the unconstrained $K$-means clustering problem due to \cite{Lloyd1982} to problem~\eqref{opt:kmean_size}. At initialization, it generates random cluster centers $\bm\zeta_k$, $k = 1, \ldots, K$. Each subsequent iteration of the algorithm consists of two steps. The first step assigns every datapoint $\bm\xi_i$ to the nearest cluster center while adhering to the prescribed cluster cardinalities, whereas the second step replaces each center $\bm\zeta_k$ with the mean of the datapoints that have been assigned to cluster $k$. The algorithm terminates when the cluster centers $\bm\zeta_1, \hdots, \bm\zeta_K$ no longer change.

{\color{black}

}

\end{document}